%% file: hydro_fund_heat.tex
\documentclass[11pt,reqno,a4paper]{amsart}

\usepackage[margin=1in]{geometry}

\usepackage{amsmath,amsthm,amssymb,mathtools}

\usepackage{times}         
\usepackage{bbm}           
\usepackage{dsfont}        
\usepackage{upgreek}       
\usepackage{enumerate}   
\usepackage{enumitem}
\usepackage[makeroom]{cancel}

\usepackage[dvipsnames]{xcolor}

\usepackage{tikz}
\usepackage{subcaption}

\usepackage{comment}

\usepackage{autonum}

\newtheorem{proposition}{Proposition}[section]
\newtheorem{theorem}[proposition]{Theorem}
\newtheorem{corollary}[proposition]{Corollary}
\newtheorem{lemma}[proposition]{Lemma}
\theoremstyle{definition}
\newtheorem{definition}[proposition]{Definition}
\newtheorem{remark}[proposition]{Remark}

\theoremstyle{definition}

\numberwithin{equation}{section}

\input{macros}

\title[Initial layer instability of the kinetic Lamb--Oseen vortex]{Initial layer instability of the kinetic Lamb--Oseen vortex}

\author{Michele Dolce}
\address{Institute of Mathematics, EPFL, Station 8, 1015 Lausanne, Switzerland}
\email{michele.dolce@epfl.ch}

\author{Isabelle Gallagher}
 \address{Université Paris Cit\'e, Sorbonne Universit\'e, CNRS, IMJ-PRG,  75013 Paris, France}
\email{isabelle.gallagher@u-paris.fr}

\begin{document}

\begin{abstract}
It is well-known that solutions to 
the incompressible Navier-Stokes system are limits  in various contexts (weak or strong), of  solutions to the Boltzmann equation, 
  when the Mach and the Knudsen numbers go to zero. 
 In particular
  the case of smooth solutions is by now rather well understood. Recent works have aimed at choosing initial data in function spaces as close as possible to those corresponding to well-posedness for the incompressible Navier-Stokes system. This paper tackles the case of measure-valued initial vorticity, in two space dimensions: in the special case when the initial vorticity is a Dirac mass, it is known that the unique solution to the Navier-Stokes system is the solution to the heat equation. We prove that the kinetic emanation of this initial vorticity (slightly smoothed out)  leads to a solution of  the Boltzmann equation which diverges in a strong way and in a very small time layer, from the expected hydrodynamic limit. 
    \end{abstract}
\maketitle
\section{Introduction}
\subsection{Setting}This article is concerned with the hydrodynamic limit of the Boltzmann equation in~$\R^2$.
Recall the Boltzmann equation for hard spheres  
\begin{align}
\label{eq:B}
\begin{cases}
    \de_tF + v\cdot \nabla_xF = \mathcal{Q}(F ,F ) \, , \qquad \text{for } (t,x,v)\in \RR_+\times \RR^2\times \RR^2\\
    F |_{t=0}=F_{\rm in}  \, ,
\end{cases}
\end{align}
with
\begin{align}
 \mathcal{Q}(f,g):=\int_{\R^2 \times \mathbb S^{1}} |v-v_*| \left[f (v'_*) g(v') - f(v_*) g(v) \right] \, d\sigma \, dv_* \,,   
\end{align}
and
$$
v':=\frac{v+v_*}{2} + \frac{|v-v_*|}{2}\, \sigma \,, \quad v'_*:=\frac{v+v_*}{2} - \frac{|v-v_*|}{2}\, \sigma\, , \quad \sigma \in \mathbb S^{1} \, .
$$
Recall that those scattering relations translate  the conservation of momentum and energy at each collision at the level of the particles. The conservation of mass is reflected in the fact that
when~$F_{\rm in} $ is a probability density, then~$F(t)$ remains a probability density for all times.

 It is well-known that to obtain the incompressible Navier-Stokes equations starting from~(\ref{eq:B}) one must perform a rescaling  in space and time (in terms of the Knudsen and  Mach numbers, denoted by~$\eps$), and   study the rescaled Boltzmann equation 
\begin{align}
\label{eq:Beps}\begin{cases}
    \de_tF^{\eps}+\eps^{-1}v\cdot \nabla_xF^{\eps}=\eps^{-2}\mathcal{Q}(F^\eps,F^\eps) \, , \qquad \text{for } (t,x,v)\in \RR_+\times \RR^2\times \RR^2\\
    F^\eps|_{t=0}=F_{\rm in}^\eps \, .
\end{cases}
\end{align}
 We assume the solution is 
 close to  a global Maxwellian, meaning
 \begin{align}
F^{\eps}=\mu+\eps\sqrt{\mu}f^\eps \, , \qquad \mu(v):=\frac{1}{2\pi}\e^{-|v|^2/2} \, .
\end{align}
Recall that~${\mu}$ and~$\sqrt{\mu}$ are stationary solutions to~(\ref{eq:Beps}). 
The system for the fluctuations $f^\eps$ is then given by
\begin{align}
\label{eq:feps}
\begin{cases}
    \de_tf^\eps+\eps^{-1}v\cdot \nabla_xf^{\eps}=\eps^{-2}Lf^\eps+\eps^{-1}\Gamma(f^{\eps},f^{\eps})\\
f^{\eps}|_{t=0}=f_{\rm in} \, ,
\end{cases}
\end{align}
where 
\begin{align}
    \label{def:LGamma}
    \Gamma(f,g):=\frac{1}{\sqrt{\mu}}\mathcal{Q}(\sqrt{\mu} f,\sqrt{\mu}g) \,, \qquad Lf:=\Gamma(\sqrt{\mu},f)+\Gamma(f,\sqrt{\mu}) \,.
\end{align}
The conservation   of mass, momentum and energy in the particle system are reflected in the $4$-dimensional kernel of $L$, namely 
\begin{align}
    {\rm Ker}(L)=\mathrm{span}\Big\{\sqrt{\mu}\big(1,v_1,v_2,\frac12|v|^2-1\big)\Big\} \, .
\end{align}
We   denote by ${\bf P}$ the projection onto   ${\rm Ker}(L)$ and we write~${\bf P}^{\perp}={\rm Id}-{\bf P}$.
Thus we introduce the hydrodynamic (or  macroscopic) part of  any function~$f$ defined on~$\R^2 \times \R^2$ as 
\begin{align}
    {\bf P} f(x,v):=\left(\rho[f](x)+u[f](x)\cdot v+\theta[f]\cdot \big(\frac12 |v|^2-1\big)\right)\sqrt{\mu}(v)
\end{align}
 and ${\bf P}^\perp f$ is  its microscopic part. We have written
\begin{equation}
 \label{eq:moments}
\begin{aligned}
\rho [f](x):= \int_{\R^2} f(x,v)\sqrt \mu(v) \, dv \, , \quad  u[f](x):= \int_{\R^2} v\,  f(x,v)\sqrt \mu (v) \, dv\, , \\
 \theta [f] (x):=  \int_{\R^2} (\frac 12|v|^2-1) f (x,v)\sqrt \mu (v) \, dv \, .
\end{aligned}
\end{equation}
For~$ f^\eps$ solving~(\ref{eq:feps}), the convergence of~$ f^\eps$ towards a hydrodynamic function whose moments are solution to the incompressible (which translates as~$ {\rm div}\, u=0$), Boussinesq (meaning $\nabla(\rho+\theta)=0$), Navier-Stokes-Fourier system
\begin{align}
\label{eq:NSF}
    \begin{cases}
        \de_t u+u\cdot \nabla u-\nu \Delta u+\nabla p=0\\
        \de_t\theta+u\cdot \nabla \theta -\nu_{\rm heat}\Delta\theta=0\\
        {\rm div}\, u=0\, ,\qquad \nabla(\rho+\theta)=0    \end{cases}
\end{align}
is known in many contexts: we shall not review all the literature here, but refer for instance to~\cite{CGT} for   references. Let us simply recall that renormalized solutions to~(\ref{eq:feps}) are known to converge in a weak sense towards Leray solutions to~(\ref{eq:NSF}) (see for instance~\cite{GSR1,GSR2,Levermore-Masmoudi,lionsmasmoudi,Saint-Raymond}), and strong, unique solutions to~(\ref{eq:feps}) are known to converge strongly towards unique solutions to~(\ref{eq:NSF}): more
precisely following \cite{BU,Ukai,GT}, for any~$\ell,k \geq 0$ let us consider the space
\begin{equation}\label{def:Xlk}
X^{\ell,k}:=\Big\{f = f(x,v) \, / \,\|f(\cdot,v)\|_{H^\ell_x}\in L^{\infty,k}_v, \, \, \sup_{|v| \ge R} \langle v\rangle^k \|f(\cdot,v)\|_{H^\ell_x} \xrightarrow[R \to \infty]{}0\Big\}\, ,
\end{equation}
where~$H^\ell_x$ is the usual Sobolev space of order~$\ell$ and
$$
L^{\infty,k}_v :=\Big\{f = f(v) \, / \, \langle v\rangle^k f \in L^\infty (\R^2)\Big\}
$$
is endowed with the norm
$$
\|f\|_{L^{\infty,k}_v}:=\sup_{v\in\R^2} \,   \langle v\rangle^k |f(v) | \,.
$$
We have used the notation~$\langle v\rangle:= \sqrt{1+|v|^2}$. 
Let us use the shorthand notation~$H^{1+}$ for the space of functions belonging to~$H^{1+\eta}(\R^2)$ for some~$\eta>0$.
It is proved in~\cite{GT}  that  if the initial data belongs to~$X^{1+,k}$ with~$k>2$ and  is well-prepared (meaning that~${\bf P} f_{\rm in} = f_{\rm in}$ and that the moments~(\ref{eq:moments}) of~${\bf P} f_{\rm in}$ satisfy the incompressibility and Boussinesq constraints) then the convergence holds in~$L^\infty(\R^+;X^{1+,k})$~; for ill-prepared initial data this remains true after an  initial time layer due to the presence of oscillating modes.  We emphasize that the convergence is shown in~\cite{GT} to be global in time whatever the size of the initial data, thanks to the fact that the 2D Navier-Stokes-Fourier system is globally well-posed for initial data in~$H^{1+}(\R^2)$.  Actually it is easy to see (we refer to   Appendix~\ref{sec:boltzmann} for more details) that in all these results, the space~$H^{1+}_x$ can be replaced by~$W_x \cap H^1_x$ where~$W_x$ is the Wiener algebra of tempered distributions whose Fourier transform is in~$L^1(\R^2)$, endowed with the norm
$$\|f\|_{W_x \cap H^1_x} := {\|\widehat f\, \|_{L^1}}
+\|f\|_{ H^1_x}\,.$$
 This is the space we shall be interested  in the following:  we define
\begin{equation}\label{eq:defX1k}
{\mathcal X}^{1,k}  :=\Big\{f = f(x,v) \, / \,\|f(\cdot,v)\|_{H^1_x \cap W_x}\in L^{\infty,k}_v, \, \, \sup_{|v| \ge R} \langle v\rangle^k \|f(\cdot,v)\|_{H^1_x\cap W_x} \xrightarrow[R \to \infty]{}0\Big\}\, .
\end{equation}
 To conclude this paragraph, we recall that the 2D Navier-Stokes system on $\RR^2$ is globally well-posed for   rougher initial data than~$H^{1+}_x(\R^2)$: Leray solutions~\cite{leray2D} are global and unique and correspond to data in~$L^2(\R^2)$, and actually there is a unique, global solution as soon as the initial vorticity is a finite Radon measure (see~\cite{cottet,GMO} for existence, and~\cite{kato,gallay-wayne3,GG,GGL} for uniqueness).  This is the setting we shall consider in this paper. 

\subsection{Statement of the main result}
 In this article, we aim at understanding how the hydrodynamic limit behaves {for scale-invariant solutions in velocity. Indeed, focusing on the equation on~$u$ in \eqref{eq:NSF} we know that solutions are invariant under the scaling $u_{\lambda}(x,t)=\lambda u(\lambda x,\lambda^2 t),\, p_{\lambda}(x,t)=\lambda^2 p(\lambda x,\lambda^2 t).$ In  2D, the associated vorticity is allowed to   be highly concentrated and form a Dirac mass.
 This class of solutions corresponds to $-1$-homogeneous initial data and} 
 a fundamental example  is the famous \emph{Lamb--Oseen vortex}. This is the solution starting with initial data 
\begin{equation}
    u_{\rm in}(x)=\frac{1}{2\pi} \frac{x^\perp}{|x|^2}\, \virgp \qquad \nabla^\perp\cdot u_{\rm in}=\omega_{\rm in}=\delta_{x=0}\, ,
\end{equation}
where $\nabla^\perp=(-\de_2,\de_1)$, and $\omega=\nabla^\perp\cdot u$ denotes the vorticity of the fluid which satisfies 
\begin{align}
\label{eq:NSV}
\begin{cases}
    \de_t\omega+u\cdot \nabla \omega-\nu\Delta\omega=0\\
    u=\nabla^\perp\psi\, , \qquad \Delta\psi=\omega\\
    \omega|_{t=0}=\omega_{\rm in}\, .
\end{cases}
\end{align}
The Lamb--Oseen vortex can be interpreted as the fundamental solution for the vorticity equation. Indeed, by the uniqueness results recalled above, we know that with $\omega_{\rm in}=\delta_{x=0}$ the system \eqref{eq:NSV} is uniquely and explicitly solved by
\begin{align}
\label{def:LO}
    \omega_{\rm LO}(x,t):=\frac{1}{\nu t}G\big(\frac{x}{\sqrt{\nu t}}\big)\,, \qquad  u_{\rm LO}(x,t):=\frac{1}{\sqrt{\nu t}}U\big(\frac{x}{\sqrt{\nu t}}\big)
\end{align}
where 
\begin{align}
    G(\xi):=\frac{1}{4\pi}\e^{-\frac{|\xi|^2}{4}}\,, \qquad U(\xi):=\frac{1}{2\pi}\frac{\xi^\perp}{|\xi|^2}\big(1-\e^{-\frac{|\xi|^2}{4}}\big)\,.
\end{align}
For the velocity field, this is a self-similar evolution starting from a $-1$-homogeneous initial data, which gives rise to a self-similar motion by the scaling properties of the Navier-Stokes equations.  Moreover, due to the azimuthal symmetry of the vortex ($u(x) = h(|x|)x^\perp$) the non-linear advection term~$u \cdot \nabla u$ in the Navier-Stokes equations is purely radial, and this is perfectly balanced by a radial pressure gradient instantly generated to preserve the incompressibility constraint. In particular, the evolution of any azimuthal velocity field is \emph{exactly} the standard heat-equation one, and in our case 
\begin{align}
\label{def:uLO}
     u_{\rm LO}(t,x) = e^{\nu t \Delta} u_{\rm in}(x) \, .
\end{align}
Moreover, 
the Lamb--Oseen vortex is the unique global attractor for the 2D Navier-Stokes dynamics \cite{gallay-wayne1} in $\mathbb{R}^2$. Therefore, the motivating question for this paper is:
\begin{quote}
    \textit{Can we obtain the Lamb--Oseen vortex as the hydrodynamic limit $\eps\to0$ from the Boltzmann equation \eqref{eq:feps} with the well-prepared initial data given as the kinetic Lamb--Oseen vortex $f_{\rm in,\rm LO}(x,v):= u_{\rm in}(x)\cdot v \sqrt{\mu}$?}
\end{quote}
This cannot be of course answered with previous available results in view of the singularity of $u_{\rm in}$ at the origin and its slow decay at infinity, which set that initial data outside of any known well-posedness class for the Boltzmann equation (we refer to Appendix~\ref{sec:boltzmann} for references on the Cauchy problem). 

Naively, given the possibility of solving explicitly the Navier-Stokes equations in this setting, one would hope nevertheless for a positive answer. However, it is not clear how to make sense of the Boltzmann equation with such initial data and, most importantly, the $-1$-homogeneity of $u_{\rm in}$ reflects in the scale invariance $$f_{\rm in, LO}(x,v)=\eps^{-1}f_{\rm in,LO}(\eps^{-1}x,v)\quad \Longrightarrow \quad f^\eps(x,v,t)=\eps^{-1}f^{1}(\eps^{-1}x,v,\eps^{-2}t).$$
This implies that one can solve \eqref{eq:feps} with the same initial data at $\eps=1$, meaning that for these  scale-invariant initial data it is not clear how one can  recover the fluid system by a  rescaling of the Boltzmann equation in a unified $\eps\to0$ limit. In addition, an informal asymptotic expansion of the form
\begin{align}
    f^{\eps}=\sum_{n=0}^\infty t^nf_{n} \,,
\end{align}
indicates that within an \textbf{initial time-layer}, with notation~(\ref{eq:moments}) and~(\ref{def:LO}),
\begin{equation}
    \|u[f^{\eps}](t)-u_{\rm LO}(t)\|_X\gg 1, \qquad \text{for } 0< t\ll {\eps^2},
\end{equation}
where $X$ is a suitable functional space accounting for the singular behavior around the origin.  An expansion as above is fundamentally different from a standard Chapman-Enskog expansion where one expands in powers of $\eps$. Unfortunately, we cannot make fully rigorous even a finite order expansion due to  the very singular nature of the initial data. However, the informal asymptotic expansion gives us the building blocks for a smoother short-time instability: we shall indeed   regularize the initial state by   introducing the regularized spatial weight 
\begin{align}
\label{def:xeps}
    |x|_\eps := \sqrt{|x|^2 + K^2\eps^2}
\end{align} 
for some large constant $K > 0$, alongside with a smooth cut-off 
\begin{align}
    \label{def:chi}
    \chi\in C^\infty_c(\mathbb{R}_+)\, , \quad \chi(r)=1 \text{ for } r\leq 1\, , \quad \chi(r)=0 \text{ for }r\geq 2\, .
\end{align}
In fact, for technical reasons we need $\chi$ to have some Gevrey regularity, but this is easily accomplished by taking a standard bump function.
We then define the regularized initial data as
 \begin{equation}\label{def_feps_in}
  f^{\eps}_{\rm in}(x,v) := \frac{1}{2\pi} \chi(\eps^\gamma |x|)\frac{x^\perp\cdot v}{|x|_\eps^2}\sqrt\mu(v) =: u_{\rm in}^\eps(x) \cdot v\sqrt{\mu}(v)\, ,
\end{equation}
where $0<\gamma< 1$ is a fixed parameter.
This is a well-prepared data that recovers the desired Lamb--Oseen velocity in the unified limit $\eps\to0$.
  \begin{remark}
By choosing~$K \gg 1$, {we shall verify} that $\eps f^{\eps}_{\rm in}$ 
remains uniformly small in $\mathcal{X}^{1,k}$.  Note that one could multiply~$u_{\rm in}^\eps(x) \cdot v$ by a more general Gaussian~$\mu^\alpha$, and with~$\alpha>1/2$ this  would ensure that~$F^\eps = \mu + \eps\sqrt{\mu}f^\eps_{\rm in}{>0}$ is a valid kinetic state. At least for $\alpha$ close to $1/2$, the same instability result below holds true upon very minor modifications in the proof. One could actually conversely choose less demanding decay on velocities using the techniques of~\cite{GMM,gervais} but this would lead to technical difficulties which are not the goal of this paper. 
  \end{remark}
 
If one expects the hydrodynamic limit to go through, the well-prepared data in \eqref{def_feps_in} with pure azimuthal macroscopic velocity would  asymptotically,   as~$\eps$ goes to zero, follow the \emph{regularized kinetic Lamb--Oseen vortex}
\begin{align}
\label{def:LOreg}
    f^\eps_{\rm LO}(t) := (e^{\nu t\Delta_x} u_{\rm in}^\eps)\cdot v\sqrt{\mu}\, , \quad \rho[f^\eps_{\rm LO}(t)]=\theta[f^\eps_{\rm LO}(t)]=0\, , \quad u[f^\eps_{\rm LO}(t)]=e^{\nu t\Delta_x}u_{\rm in}^\eps(x) \, .
\end{align}
In a short time layer $t \ll \eps^2$, the macroscopic velocity enjoys the standard Taylor expansion of the heat semigroup, so that
\begin{align}
\label{eq:TayloruLO}
    u[f^\eps_{\rm LO}](t,x)=u_{\rm in}^\eps(x)+\nu t\Delta_xu_{\rm in}^{\eps}(x)+\mathcal{O}\Big(\frac{(\nu t)^2}{\eps^5}\Big) \, ,
\end{align}
where the $\mathcal{O}$ term is measured in a suitable functional space.

The main result of this paper is  that the true macroscopic moments emanating from the kinetic evolution fundamentally fail to follow this expected Navier-Stokes state in the initial time-layer. In the statement, we denote   the radial and tangential components of the macroscopic velocity as 
      $$u[f^\eps] = u_{\rm rad}[f^\eps] e_r + u_{\rm tan}[f^\eps] e_\theta\, . $$
 \begin{theorem}\label{thm:result}
 Let $k>2$. There exists a universal constant $C_\star>0$, a large enough constant~$K>0$, a small enough constant~$\delta>0$, and constants $0<a_0<b_0<\infty$, $C,c>0$, all independent of
$\eps$, 
   such that the following holds. Let~$f^{\eps}(t)$ be the unique solution to \eqref{eq:feps} associated with the initial data $f^{\eps}_{\rm in}$ defined in~\eqref{def_feps_in}, and let $f^\eps_{\rm LO}(t)$ be as in~\eqref{def:LOreg}. Let $$T_\eps:=\delta \frac{\eps^2}{C_\star|\ln(\eps)|^{ 6}} \, \cdotp$$ 
  Then   for all~$\eps$ small enough, and all~$t \leq T_\eps$,
  \begin{align}
      \|(f^\eps-f^{\eps}_{\rm LO})(t)\|_{\mathcal{X}^{1,k}}\geq C\frac{t}{\eps^3} \, \cdotp
  \end{align}
  Moreover, the hydrodynamic variables deviate from the fluid target : for all~$t\leq T_\eps$, defining the set~$    A_\eps:=\big\{x\in\mathbb R^2\, / \,  a_0\eps\le |x|\le b_0\eps\big\}$, then for all~$x \in A_\eps$:    \begin{enumerate}
      \item the radial and tangential components of the macroscopic velocity satisfy
      \begin{equation}
       \label{eq:bound_u}
      \Big|(u_{\rm tan}[f^{\eps}]-u_{\rm tan}[f^{\eps}_{\rm LO}])(t,x)\Big| \geq c \frac{t}{\eps^3}\, \virgp \qquad \Big|u_{\rm rad}[f^{\eps}](t,x)\Big|\geq c \frac{t^2}{\eps^5} \, \virgp
  \end{equation}
  \item the incompressibility constraint is violated in the sense that
  \begin{align}
      \label{eq:bound_div} \big|\nabla_x \cdot u[f^{\eps}](t,x)\big| &\geq c \frac{t^2}{\eps^6} \, \virgp
            \end{align}
      \item the density and temperature   experience the instability 
      \begin{equation}
      \label{eq:bound_rho}
                \big|\rho[f^{\eps}](t,x)\big|+\big|\theta[f^{\eps}](t,x)\big| \geq c \frac{t^3}{\eps^7} \,  \cdotp
      \end{equation}
  \end{enumerate}
  \end{theorem}   
The lower bounds in \eqref{eq:bound_u}--\eqref{eq:bound_rho} encode the breakdown of the validity of the standard 2D Navier--Stokes--Fourier system as the target dynamics, even for well-prepared data. This is reflected in the strong violation of the expected incompressible fluid behavior quantified in \eqref{eq:bound_u}--\eqref{eq:bound_rho}. We believe that a natural explanation is provided by the different time scales on which the kinetic and fluid constraints are established, namely the inherent finite time required for kinetic relaxation, as opposed to the instantaneous regularization and pressure formation that preserve incompressibility in \eqref{eq:NSF}.

To explain this, observe that, even though the initial datum is classically well prepared for the incompressible Navier--Stokes equations, the regularized velocity satisfies $|u_{\rm in}^\eps|\sim\eps^{-1}$ near the origin. Consequently, the ``local Mach number'' is formally of order one, namely
\[
{\rm M}_{\rm loc}
=
\frac{\eps|u_{\rm in}^\eps|(x)}{\mathsf{c}_s}
\approx
\mathcal{O}(1) \quad \text{ for } \quad |x|\approx \eps,
\]
where~$\mathsf{c}_s$ denotes the sound speed, suggesting that one should not expect an incompressible regime to hold uniformly near the origin\footnote{Interestingly enough, approximating highly concentrated vortices by Dirac masses in 2d inviscid fluids also poses concerns when looking at the incompressibility constraint, as originally noted by Onsager, see \cite{Eyink06}.  In the viscous fluid system, diffusion uniquely selects  an incompressible evolution that regularizes the concentration. On the other hand, the kinetic dynamics considered here initially selects a different, compressible state.}.
In the incompressible fluid system, the radial acceleration generated by $(u^\eps\cdot\nabla)u^\eps$ is instantaneously compensated by the pressure gradient, so that the velocity remains divergence free. Our kinetic expansion shows instead that this macroscopic pressure balance is not established at first order in $t$ within the initial time layer, since the first correction is purely microscopic. The macroscopic kinetic response appears only at subsequent orders, through the transport of the stress generated by the microscopic dynamics. This mechanism produces a radial velocity of size $\mathcal{O}(t^2/\eps^5)$, as quantified in \eqref{eq:bound_u}, which may be interpreted as a macroscopic large centrifugal effect driving the dynamics radially outward. Its spatial variation yields the divergence lower bound $\mathcal{O}(t^2/\eps^6)$ in \eqref{eq:bound_div}, showing that the incompressibility constraint is violated inside the kinetic core layer. Through the local mass and energy balances, this divergence subsequently generates density and temperature fluctuations of size $\mathcal{O}(t^3/\eps^7)$, as quantified in \eqref{eq:bound_rho}.

In summary, the concentrated vortex develops a genuinely compressible kinetic initial layer before the usual incompressible viscous state has time to emerge. While the Navier--Stokes evolution immediately regularizes the vortex through diffusion and the pressure enforces incompressibility, the kinetic evolution first undergoes a remarkably different compressible dynamics on the initial time-layer $t\ll\eps^2$, before the collisional regime is expected to establish a genuinely incompressible viscous fluid state.

  \begin{remark}
    For our instability,  it is crucial to choose the regularization to be $\eps$-dependent. In fact, by regularizing uniformly in $\eps$ one enters in the regime studied in \cite{Kimpv}, where the authors perform a suitable hydrodynamic limit towards more general point vortex interactions, as in e.g. \cite{gallayARMA}. However, in \cite{Kimpv} the authors take as a kinetic initial data the evolution of the 2D Navier-Stokes after an $\mathcal{O}(1)$-time, meaning that the leading order heat-evolution has removed the singular behavior that we actually exploit in the initial time-layer. Given our short-time regime and the extremely localized instability, we believe it is feasible to extend our result to a collection of initial point-vortices as in \cite{Kimpv}. Moreover, we use the specific form of our initial data only to verify that we satisfy certain assumptions required to construct the approximate solution in Section \ref{sec:expansion}, but the arguments here easily generalize  to a class of initial data satisfying the desired hypotheses.
    \end{remark}
 \begin{remark}
        The other fundamental ingredients of our instability result are the fact that in the $\eps\to 0$ limit we target a $-1$-homogeneous initial data and that the leading order evolution of the Navier-Stokes equations is dictated by the heat flow. For small $-1$-homogeneous data, Brandolese \cite{brandolese2009} provided precise expansions valid as $t\to0$ aligning with this expectation, and similar estimates are expected for the large scale-invariant solutions constructed by Jia and \v{S}ver\'ak in \cite{JiaSverak} for the 3D Navier-Stokes equations. We comment more about this in Section \ref{sec:3D_perspectives}.
    \end{remark} 
  \begin{remark}
       The closing argument of the  instability result relies on a fixed-point method, and a minor adaptation of the argument would  allow for small, well-prepared perturbations of the initial data. 
    \end{remark} 
 \begin{remark}
       For the sake of simplicity we  have written the proof in the case of hard spheres, but it seems feasible to extend it to more general collision kernels. In fact, the case of hard potentials is technically more challenging compared to the soft ones in our construction of the approximate solution. This is because we treat perturbatively collisions and therefore for hard potentials one has to face  worst  velocity weight losses. 
    \end{remark} 
     \begin{remark}
       The same result would hold in a three-dimensional setting provided the~$x$ variable remains 2D, up to minor modifications in the proof. An axisymmetric setting in space would instead be physically more interesting, but technically more involved, so we leave this for future work.
    \end{remark} 
\subsection{Plan of the paper}
The coming section presents a naive expansion, which indicates the relevant regimes at play. Regularity issues prevent one from arguing directly using that expansion, and Section~\ref{sec:expansion} is devoted to the precise construction  of the approximate solution, by writing recursively an expansion up to an order~$M$ which will be chosen of size~$|\ln (\eps)|$. Assuming the expansion does provide an approximate solution, its form enables to derive all the instability lower bounds of Theorem~\ref{thm:result}. 
Section~\ref{eq:correction} is devoted to the proof that the   expansion does indeed provide an approximate solution to the rescaled Boltzmann equation: a large~$M$ is crucial here to close the nonlinear estimates. {In Section~\ref{sec:3D_perspectives} we comment on how to extend the kinetic instability framework to generic 2D and 3D situations.} Finally the appendix presents some classical results on the Boltzmann equation, adapted to the current functional framework.

\subsection{Notation convention}
We write~$C\gg1$ if~$C$ is larger than a large number, independent of~$\eps$. Similarly   we write~$A\lesssim B$ for $A \leq CB$ with~$C $ a universal constant.

\section{The naive asymptotic expansion}
\label{sec:naive}
In this section we present the informal arguments to construct the building blocks of our short-time instability.  
First of all, given the self-similar nature of the evolution in the Navier-Stokes equations in \eqref{def:uLO}, it seems natural to introduce a self-similar ansatz in the space-time variables for \eqref{eq:feps}, namely
\begin{align}
f^\eps(t,x,v)=\frac{1}{\sqrt{t}}\varphi^{\eps}\big(t,\frac{x}{\sqrt{t}},v\big)\,, \qquad \xi=\frac{x}{\sqrt{t}}\,\cdotp
\end{align} 
Then, a direct computation shows that $\varphi^\eps$ satisfies 
\begin{align} 
\label{eq:ssphieps}
t\de_t\varphi^\eps-\frac12(1+\xi\cdot\nabla_\xi)\varphi^\eps+\frac{\sqrt{t}}{\eps}\Big(v\cdot \nabla_\xi \varphi^\eps-\Gamma(\varphi^{\eps},\varphi^{\eps})\Big)=\frac{t}{\eps^2}L\varphi^\eps\,.
\end{align}
By the $-1$-homogeneity of $u_{\rm in}$, note that the initial datum for $\varphi^\eps$ is the same, that is
\begin{equation}
\label{def:phi0}
    \varphi^{\eps}|_{t=0}(\xi,v)=u_{\rm in}(\xi)\cdot v\sqrt{\mu }=\frac{1}{2\pi}\frac{\xi^\perp\cdot v}{|\xi|^2}\sqrt{\mu }\, .
\end{equation} 
Looking at \eqref{eq:ssphieps}, we readily notice three different regimes:
\begin{itemize}[itemsep=2pt]
    \item \textbf{Initial time-layer} $t\ll \eps^2$: within these time scales, the leading order operator is $v\cdot \nabla_\xi-\Gamma(\cdot,\cdot)$. Indeed, since our initial datum is singular at the origin, both the transport and the nonlinear collisions make the singularity worse at the origin, whereas $L$ (and $1+\xi\cdot\nabla_\xi$ as well) do not change said spatially singular behavior. Unless very special cancellations are at play, we can therefore expect that the operators with the $\sqrt{t}/\eps$ scaling are indeed leading the dynamics. For our regularized kinetic Lamb--Oseen vortex, we characterize precisely how the singularity at the origin evolves within this time-layer, exhibiting a completely different behavior as opposed to the expected incompressible fluid state.
    \item \textbf{Hydrodynamic regime} $t\gg \eps^2$: in this time-scale, assuming that the initial datum had enough time to smooth out, one should expect to enter a collision-dominated regime where $L$ dominates over transport and nonlinear collisions. Thus, here one should expect to run a standard perturbative approach (following for instance the Kato spectral approach in \cite{Gervais24}  based on the Ellis and Pinsky decomposition~\cite{Ellis-Pinsky}) to prove that the macroscopic variables associated with solutions to \eqref{eq:ssphieps} converge as $\eps \to 0$ to the Navier-Stokes-Fourier system \eqref{eq:NSF} in self-similar variables.
    \item \textbf{Transition time-layer} $t\approx \eps^2$: here is where the transition from the initial time-layer to the hydrodynamic regime happens. Understanding precisely how the dynamics realize this transition seems complicated because there are no clear scale separations to exploit. It would be interesting to address this point to fully characterize the evolution of the regularized kinetic Lamb--Oseen vortex, which could serve as the basis to understand how to rigorously perform the hydrodynamic limit for measure-valued initial data that are allowed in the 2D Navier-Stokes equations. 
\end{itemize}
\begin{remark}
    It is expected that one  can perform the hydrodynamic limit by decoupling the regularization parameter from $\eps$ (similar in spirit to \cite{Kimpv}), but this procedure only recovers the singular data \emph{after} the hydrodynamic limit has been performed, completely bypassing the initial time-layer which can hide some interesting physical phenomena related to different time-scales for the mesoscopic and macroscopic scales.
\end{remark}
In this paper, we focus exclusively on the initial time-layer $t \ll \eps^2$. While the self-similar equation \eqref{eq:ssphieps} is a nice tool to detect these three distinct regimes, performing the actual expansion in powers of $\sqrt{t}/\eps$ within self-similar variables is unnecessarily cumbersome. For our purposes, it is more direct to switch back to standard variables and perform a standard Taylor expansion in time directly on the Boltzmann equation \eqref{eq:feps}. 
To construct an informal asymptotic expansion, we assume that 
\begin{align}
\label{def:fheuristic}    f^{\eps}(t,x,v)\sim f_{\rm in, \rm LO}(x,v)+\sum_{n=1}^Mt^nf_{n}^\eps(x,v)\, ,
\end{align}
recalling that~$f_{\rm in, \rm LO}(x,v):= u_{\rm in}(x)\cdot v \sqrt{\mu}
$. Let us focus only on matching orders in $t$, in which case by plugging the ansatz above into \eqref{eq:feps} we find the following:
\begin{enumerate}[label=\roman*)]
    \item \textbf{Order $t^0$}: since we are starting with well-prepared data, there holds~$L(f_{\rm in, \rm LO})=0$. Thus at this order we must solve 
    \begin{align}
        f_1^\eps=\eps^{-1}\big(-v\cdot \nabla_x f_{\rm in, \rm LO}+\Gamma(f_{\rm in, \rm LO},f_{\rm in, \rm LO})\big)\, .
    \end{align}
    \item \textbf{Order $t^1$}: a direct computation shows that to match at order $t$, one finds that 
    \begin{align}
2f_2^\eps=\eps^{-2}L(f_1^\eps)+\eps^{-1}\big(-v\cdot \nabla_x f_{1}^\eps+\Gamma(f_{\rm in, \rm LO},f_{1}^\eps)+\Gamma(f_{1}^\eps,f_{\rm in, \rm LO})\big)\, .
    \end{align}
    \item \textbf{Order $t^{n-1}$}: generalizing this process to solve for $f_n^\eps$, one finds the recursive relation 
    \begin{align}
        n f_n^\eps = \eps^{-2}L(f_{n-1}^\eps) + \eps^{-1}\bigg(-v\cdot \nabla_x f_{n-1}^\eps + \sum_{k=0}^{n-1}\Gamma(f_k^\eps, f_{n-1-k}^\eps)\bigg)\,,
    \end{align}
    where we use the convention $f_0^\eps := f_{\rm in, LO}$.
\end{enumerate}
Now, given the singular nature of our initial data, we should not only separate orders in $t$ but we should also account for the singular behavior at the origin. For instance, 
\begin{align}
    \frac{1}{\eps |x|^2}\approx \frac{1}{\eps^2|x|} \quad \text{for } |x|\approx \eps.
\end{align}
To this end, let us only look at the expected behavior around the origin. For the initial data we have $|f_{\rm in, \rm LO}|\approx |x|^{-1}$. Then, since $|\nabla_xf_{\rm in, LO}|\approx |x|^{-2}$ and $|\Gamma(f_{\rm in, LO},f_{\rm in, LO})|\approx |x|^{-2}$ we expect that 
\begin{align}
    |f_{1}^\eps|\approx \eps^{-1}|x|^{-2}, \quad \Longrightarrow \quad |f_2^{\eps}|\approx \eps^{-3}|x|^{-2}+\eps^{-2}|x|^{-3}.
\end{align}
We then start recognizing a pattern. At each step we are effectively  multiplying the spatial singularity by an additional factor of $\eps^{-1} |x|^{-1}$. Evaluated near the vortex core where $|x|\approx \eps$, each subsequent order systematically loses an $\mathcal{O}(\eps^{-2})$ factor. Therefore, the behavior of the term $f_n^\eps$ is actually of order 
\begin{equation}
    |f_{n}^\eps|\approx \eps^{-2n-1}
\end{equation}
and 
\begin{align}
    |f^\eps|\approx \eps^{-1}\Big(1+\sum_{n=1}^M(t/\eps^2)^n \Big).
\end{align}
Hence, we see that in the initial time-layer $t\ll \eps^2$ we have a clear separation of orders for $|x|\approx \eps$, and we can hope to approximate the solution up to a given arbitrary order.  This asymptotic scaling at $|x|\approx \eps$ is in fact a direct consequence of the scale invariance of the initial data. Indeed, since $f_{\rm in, LO}(x,v)=\eps^{-1}f_{\rm in, LO}(\eps^{-1}x,v)$ we readily see that 
\begin{equation}
    f^{\eps}(t,x,v)=\eps^{-1}f^{1}\Big(\frac{t}{\eps^2},\frac{x}{\eps},v\Big)\sim f_{\rm in,LO}(x,v)+\frac{1}{\eps}\sum_{n=1}^M\Big(\frac{t}{\eps^2}\Big)^nf_{n}^1\Big(\frac{x}{\eps},v\Big),
\end{equation}
and solutions of our recurrence relation for $\eps=1$ are an $\mathcal{O}(1)$ at $|x|\approx \eps$.

Moreover, by a direct inspection of $f_1^\eps$, it is not hard to see that the first-order correction is indeed purely microscopic, since 
\begin{align}
    {\bf P}f_1^\eps=0\, .
\end{align}
This follows because our initial datum is incompressible and has zero density and temperature. Hence, the odd velocity parity perfectly cancels the macroscopic contributions from the transport term. Instead, for the nonlinear part, we simply exploit the standard mass, momentum, and energy conservation property ${\bf P}\Gamma(f,f)=0$. This  cancellation is the exact kinetic manifestation of the ``missing pressure'' discussed in the introduction.  This confirms that in the informal asymptotic expansion we have no macroscopic contributions at order $t$, thus indicating that we readily observe a difference with respect to the standard Lamb--Oseen vortex, whose first non-trivial order goes like
\begin{align}
    \nu t|\Delta_xu_{\rm in}|\approx  \nu t|x|^{-3} \approx \nu t\eps^{-3} \quad \text{for }|x|\approx \eps.
\end{align}
Because the first macroscopic kinetic correction occurs at order $t^2$ (scaling as $t^2|u[f^\eps_2]| \approx t^2\eps^{-5}$), for $|x|\approx \eps$ we get 
\begin{align}
  \big  |u[f^\eps]-u_{\rm LO}\big|\approx\big |\nu t \Delta_x u_{\rm in} - t^2 u[f^\eps_2]\big| \approx  t\eps^{-3} \quad \text{ since } t^2\eps^{-5} \ll  t\eps^{-3} \text{ when } t\ll \eps^2.
\end{align}
This means the expected viscous behavior from the Navier-Stokes equations is fundamentally wrong for the kinetic evolution at this time scale. 

To make this intuition rigorous, we must first regularize the initial data so as to have a clear separation of orders for small $|x|$. Then, we have to precisely quantify the size of the errors in a suitable functional space where the Boltzmann equation is well-posed, and in particular, we must carefully track the size of the error when we sum everything up. This is performed in the next section, but the whole construction is based on the basic intuition built from the simple informal calculations above.

\section{Expansion  emanating from the regularized kinetic Lamb--Oseen vortex}\label{sec:expansion}
Inspired by the structure of the approximate solution found in Section \ref{sec:naive}, we now proceed with our expansion for the regularized data. The idea is simply to smooth out each $|x|$ by replacing it with $|x|_\eps$ as defined in \eqref{def:xeps}. Recall also the cut-off defined in \eqref{def:chi}. Then, we introduce the macroscopic variables 
\begin{equation}
\label{def:rhoz}
    r(x):=\frac{x}{|x|_\eps} \in B_1(0)\, , \qquad z(x):=\eps^\gamma x \in \mathbb R^2 \, , 
\end{equation}
and we rewrite our initial regularized kinetic Lamb--Oseen vortex as 
\begin{align}
\label{def:rhofin}
f^\eps_{\rm in}(x,v)=\frac{1}{|x|_\eps}\Phi_0\big(r(x), v, z(x)\big), \quad \Phi_0(r,v,z):=\frac{1}{2\pi}(r^\perp \cdot v) \sqrt{\mu}(v) \chi(|z|)\, .
\end{align}
Note that $\Phi_0(r,v,z)$ is a smooth and bounded function in all variables, has Gaussian decay in $v$ and is supported in $| z |\le 2$. Thus, the singular behavior has been completely isolated within the prefactor~$|x|_\eps^{-1}$. Since $|x|_\eps \geq K\eps$, the leading order terms when measuring the $\mathcal{X}^{1,k}$-norm defined in~(\ref{eq:defX1k}) are indeed dictated by the power of the singularity at the origin. Furthermore, as $\eps \to 0$, the cut-off radius~$\eps^{-\gamma} \to \infty$, ensuring that we recover the Lamb--Oseen vortex in the unified $\eps\to 0$ limit.

Looking at the structure of the approximate solution found informally in Section \ref{sec:naive}, if we simply start the recurrence relation from \eqref{def:rhofin} we will carry over higher powers of $\eps^{-1}$ and $|x|_\eps^{-1}$ at each step. To rigorously construct the approximate solution, we must carefully track that singular behavior and estimate precisely the error at each approximation step. 
 In particular, in the initial time-layer $t\ll\eps^2$ we aim at rigorously establishing this approximation up to an arbitrary finite order $M$, and we need to track all the dependencies on $M$ and $\eps$ since later it will be necessary to take $M\approx |\ln (\eps)|$.

We start by introducing the technical framework needed to define precisely the recurrence relation defining our expansion. 

\subsection{Toolbox}
\label{sec:Toolbox}
To carefully control the velocity moments generated by the collision and transport operators, we allow for quantitative losses in the Gaussian decay (analogous to the shrinking of the radius of analyticity in Cauchy-Kovalevskaya type theorems\footnote{this is also used for instance in the study of the BBGKY hierarchy for the Boltzmann equation -- see~\cite{Lanford,GSRT,Fougeres} }). Let $M \ge 1$ be a fixed integer representing the maximum order of our expansion. We define 
\begin{align}
\label{def:losing_weights}
    \beta_n:=\frac14 - \frac{n}{32M}\, \virgp \qquad \mu_{n}(v) := e^{-\beta_n |v|^2/4}\,, \qquad N_M:=2M+4\,.
\end{align}
We then introduce the following.
\begin{definition}\label{def:Peps}
Let $n \in \{0, \dots, N_M-2\}$. We say that a function $g=g(x,v)$ belongs to the class $\mathcal{P}_\eps(n)$ if it can be written as 
\begin{equation}\label{eq:Peps_form}
    g(x,v) = \frac{1}{\eps^n |x|_\eps^{n+1}} \sum_{j=0}^{n} \Big(\frac{|x|_\eps}{\eps}\Big)^j\Phi_{n,j}\big(r(x), v, z(x)\big) \, ,
\end{equation}
where $r(x),z(x)$ are defined in \eqref{def:rhoz}, and each \emph{profile} $\Phi_{n,j}(r,v,z)$ is smooth in $(r, z) \in B_1(0) \times \R^2$, compactly supported in $|z| \le 2$ and satisfies the following: there exists a constant $\mathfrak{g}_n>0$ such that
\begin{equation}\label{eq:profile_norm}
    \sum_{j=0}^n \sum_{|\alpha| + |\beta| \le N_M - n} \frac{\eps^{\gamma |\alpha|}}{\alpha! \beta!} \sup_{r \in B_1(0), \, z \in B_2(0)}  \big| \nabla_z^\alpha\nabla_r^\beta  \Phi_{n,j}(r,v,z) \big| \leq \mathfrak{g}_n \mu_{n}(v) \, ,
\end{equation}
where $\alpha=(\alpha_1,\alpha_2),\, \beta=(\beta_1,\beta_2)$ are multi-indices. We refer to $\mathfrak{g}_n$ as the \emph{profile bound} of $g$.
\end{definition}
Note that we require the definition for~$n$ ranging up to $2M+2$. This is because the nonlinear residual contains a derivative and products of two terms of order at most~$M$, meaning that  at least $2M+1$ is required to define properly the residual. Then, with $N_M=2M+4$ we are left with at least two spatial derivatives under control in \eqref{eq:profile_norm} when $n=N_M-2$, which is useful to exploit standard Sobolev embeddings when necessary.

Moreover, by \eqref{eq:profile_norm}, the profiles trivially exhibit a vanishing tail in velocity, namely for any $k \ge 0$,
\begin{equation}\label{eq:vanishing_tail_def}
    \lim_{R \to \infty} \sup_{|v| \ge R} \langle v\rangle^k \big\| \Phi_{n,j}(\cdot,v,\cdot) \big\|_{W^{1,\infty}_{r,z}} = 0 \, .
\end{equation}
This ensures that functions in $\mathcal{P}_\eps(n)$ satisfy a necessary condition to belong to the $\mathcal{X}^{1,k}$ spaces (using also the compact support assumption in the $z$-variable which deals with the~$H^1_x$ part of the norm; the~$W_x$ part is dealt with as in the computations done below).

\begin{remark}
\label{def:profilef0}
As an example, we readily observe that $f_{\rm in}^\eps$ defined in \eqref{def:rhofin} belongs to $\mathcal{P}_\eps(0)$. Indeed by our choice of $\beta_0 = 1/4$, we have $\mu_{0}(v) = e^{-|v|^2/16}$. The initial profile is exactly  \begin{equation}
    \Phi_0(r,v,z) = \frac{1}{(2\pi )^\frac32} (r^\perp \cdot v) e^{-|v|^2/4} \chi(|z|)  =\frac{1}{(2\pi )^\frac32} (r^\perp \cdot v) e^{-3|v|^2/16} \chi(|z|) \mu_{0}(v) \, .
\end{equation}
To evaluate its profile bound, we must control the $\eps$-weighted $r,z$ derivatives up to order $2M+4$. First, note that $\displaystyle \sup_{v}|v|e^{-3|v|^2/16}
=\sqrt{8/(3e)}
$, and therefore we can uniformly bound the factor in $v$. Then, since~$\Phi_0$ is linear in $r$, any $r$-derivatives of order at least two vanish identically. Hence, we can explicitly define
\begin{equation}\label{eq:boundf0}
    \mathfrak{f}_0 := \frac{1}{(2\pi)^\frac32 } \sqrt{\frac{8}{3e}} \sum_{|\alpha| + |\beta| \le N_M} \frac{\eps^{\gamma|\alpha|}}{\alpha!\beta!} \big\| \nabla_z^\alpha \chi(|z|) \big\|_{L^\infty} \big\| \nabla_r^\beta r^\perp \big\|_{L^\infty(B_1(0))} \, .
\end{equation}
Because the sum over $\beta$ only contains terms for $|\beta| \le 1$, it produces a constant factor. We can then fix the cut-off $\chi$ as a bump function that belongs to a suitable Gevrey class so that $\|\partial_z^\alpha\chi\|_{L^\infty}\leq C_\chi^{|\alpha|} (|\alpha|!)^\sigma$ for some $\sigma>1$. The terms in the sum over $\alpha$ are therefore bounded by $(C_\chi\eps^\gamma )^{|\alpha|} (|\alpha|!)^{\sigma-1}$. Since we will take at most $M=\mathcal{O}(|\ln\eps|)$, the growth of the Gevrey factorial is easily absorbed by $\eps^\gamma$ for $\eps$ sufficiently small, meaning the finite sum defining $\mathfrak{f}_0$ is uniformly bounded by an $\mathcal{O}(1)$ universal constant provided that $M\leq C_0|\ln (\eps)|.$
\end{remark}

\begin{remark}
\label{rem:Wienernorm}
    Note that from the definition above we immediately deduce that the $L^\infty_x$ norm of a term in the class $\mathcal{P}_\eps(n)$ is of order $\eps^{-(2n+1)}$, which is what we expect from the naive expansion. However, to compute the $W_x$ norm it is more convenient to first set
    \begin{equation}
        y_\eps(x) := \frac{x}{K\eps}\, \virgp \quad |x|_\eps = K\eps \langle y_\eps(x)\rangle, \quad r(x) = \frac{y_\eps(x)}{\langle y_\eps(x)\rangle}\, \virgp\quad z(x) = K\eps^{\gamma+1}y_\eps(x) \, .
    \end{equation}
    Then consider a typical term in $\mathcal{P}_\eps(n)$ given by $\displaystyle g_{n,j} := \frac{1}{\eps^{n+j} |x|_\eps^{n+1-j}} \Phi_{n,j}$. We can rewrite it as 
    \begin{align}
        g_{n,j}(x,v) = \frac{1}{\eps^{2n+1}}\frac{1}{K^{n+1-j}} \Psi_{n,j}(y_\eps(x),v) \, ,
    \end{align}
    where 
    \begin{align}
    \label{def:Psinj}
        \Psi_{n,j}(y,v) = \frac{1}{\langle y \rangle^{n+1-j}}\Phi_{n,j}\Big(\frac{y}{\langle y \rangle}, v, K\eps^{\gamma+1}y\Big) \, .
    \end{align}
    Using the scaling property $\mathcal{F}_x[g(\cdot/\alpha)](\xi) = \alpha^2 \mathcal{F}_y[g](\alpha\xi)$ of the Fourier transform in 2D and the change of variables $\eta=\alpha\xi$ we readily get
    \begin{align}
        \|g_{n,j}(\cdot,v)\|_{W_x} = \frac{1}{\eps^{2n+1}K^{n+1-j}}\|\Psi_{n,j}(\cdot, v)\|_{W_y} \, .
    \end{align}
    One is then left with the control of the $L^1$ norm of the Fourier transform of a typical term as in \eqref{def:Psinj}, which we can bound uniformly thanks to the spatial decay of order at least $\langle y\rangle^{-1}$ and the required regularity of~$\Phi_{n,j}$. This will be shown precisely in Lemma \ref{lem:analytic_bounds}.
\end{remark}
To quantify precisely the dependence on $n$ and $\eps$ of the $\mathcal{X}^{1,k}$-norm, we have the following.
\begin{lemma}\label{lem:analytic_bounds}
For any $g \in \mathcal{P}_\eps(n)$ with $0 \le n \le N_M-2$ and associated profile bound $\mathfrak{g}_n$, then~$g  $ belongs to~$ \mathcal{X}^{1,k}$ for all~$k \geq 0$. Moreover, fix $K\geq1 $, $ 0<\gamma< 1$ and $k\ge 0$. Then there exists  $\eps_0(K,\gamma)=\eps_0$ and $C_k>0$, with $C_k$ independent of
$\eps,K,n,M$, such that for all $0<\eps<\eps_0$ there holds
\begin{equation}
    \| g \|_{\mathcal{X}^{1,k}} \leq C_k (n+1)^3 \mathfrak{g}_n \frac{1}{K \eps^{2n+1}} \, \cdotp
\end{equation}
\end{lemma}
\begin{proof}
Let $g \in \mathcal{P}_\eps(n)$. By definition, it is a finite sum of terms of the form 
$$g_{n,j}(x,v) :=  \frac{1}{\eps^{n+j} |x|_\eps^{n+1-j}} \Phi_{n,j}\big(r(x),v,z(x)\big) \, ,$$
where~$j$ runs from~$0$ to~$n$. 

For the $H^1_x$ norm, we first bound the $L^2_x$ norm. By the properties of the profile $\Phi_{n,j}$,  the pointwise bounds are controlled by  $\mathfrak{g}_n\mu_n(v)$. It is then enough to study what happens when we integrate negative powers of $|x|_\eps$, and we consider two cases:
\begin{itemize}
    \item[$\diamond$] If $j<n$ then
    \begin{align}
       \|g_{n,j}(\cdot, v)\|_{L^2_x}\leq \frac{\mathfrak{g}_n \mu_{n}(v)}{\eps^{n+j}} \big\||x|_\eps^{-(n+1-j)}\big\|_{L^2_x}= \sqrt{\frac{\pi}{n-j}} \frac{\mathfrak{g}_n \mu_{n}(v)}{K^{n-j}\eps^{2n}} \, \cdotp
    \end{align}
    This is better than the stated bound since $\eps^{-2n} \ll \eps^{-2n-1}$ and~$K>1$.
    \item[$\diamond$] If $j=n$, we exploit the compact support on the $z$ variable to get that
    \begin{align}
      \|g_{n,n}(\cdot, v)\|_{L^2_x} \leq \frac{\mathfrak{g}_n \mu_{n}(v)}{\eps^{2n}}\Big(\int_{|x| \le 2\eps^{-\gamma}} \frac{1}{|x|_\eps^{2}} \, dx \Big)^\frac12\leq C |\ln \eps|^\frac12 \frac{\mathfrak{g}_n \mu_{n}(v)}{\eps^{2n}} \, \cdotp
    \end{align}
    This is again better  since $|\ln \eps|^\frac12 \eps^{-2n} \ll \eps^{-2n-1}$ for small $\eps$.
\end{itemize}
As expected, the leading order scaling therefore arises from the $L^2_x$ norm of $\nabla_x g_{n,j}$. Since
\begin{align}
\label{eq:trivial x rho R}
    \nabla_x |x|_\eps = r \,  \, , \qquad \nabla_x r_i = \frac{1}{|x|_\eps} \big(e_i - r_i r\big) \, , \qquad \nabla_x z = \eps^\gamma \mathbb{I} \, ,
\end{align}
where $\mathbb{I}$ is the $2\times 2$ identity matrix, applying the chain rule we see that
\begin{align}
    \nabla_x g_{n,j} = \frac{1}{\eps^{n+j} |x|_\eps^{n+2-j}} \bigg[ -(n+1-j)r \Phi_{n,j} + (\mathbb{I} - r\otimes r)\nabla_r \Phi_{n,j} + \eps^\gamma |x|_\eps \nabla_z \Phi_{n,j} \bigg] \, .
\end{align}
For the first two terms, since $n+2-j\geq 2$, they are square-integrable even without exploiting the cut-off.  Thus
\begin{align}
    \frac{1}{\eps^{n+j}}\Big\||\cdot|_\eps^{-(n+2- j)}r \Phi_{n,j}(\cdot, v)\Big\|_{L^2_x}\leq C\frac{\mathfrak{g}_n \mu_{n}(v)}{K\eps^{2n+1}} \, ,
\end{align}
where $C>0$ is a universal constant, and an analogous bound holds for the term involving $\nabla_r\Phi_{n,j}$. For the term involving $\eps^\gamma\nabla_z\Phi_{n,j}$, observe that its size is determined by the $L^2$-norm of $\eps^{-(n+j)}|x|_\eps^{-(n+1-j)}$, which is controlled exactly as above. Summing over the $n+1$ terms, the $H^1_x$ bound scales as 
\begin{align}
\label{bd:H1x}
    \|g\|_{H^1_x}\lesssim (n+1) \frac{\mathfrak{g}_n \mu_{n}(v)}{K\eps^{2n+1}}\, \cdotp
\end{align}

Finally, we turn our attention to controlling the $W_x$-norm. As observed in Remark \ref{rem:Wienernorm}, we have that 
\begin{align}
    \eps^{2n+1}K^{n+1-j}\|g_{n,j}(\cdot, v)\|_{W_x} 
    = \int_{\mathbb{R}^2}|\widehat{\Psi}_{n,j}(\eta,v)| d\eta \, .
\end{align}
We then split into low frequencies $|\eta|\leq 1$ and high frequencies $|\eta|>1$. For the latter, by applying the Cauchy-Schwarz inequality we get  
\begin{align}
    \int_{|\eta|>1}|\widehat{\Psi}_{n,j}(\eta,v)| d\eta \lesssim \Big\||\cdot|^2\widehat{\Psi}_{n,j}(\cdot,v)\Big\|_{L^2_\eta} \lesssim \|\Delta_y\Psi_{n,j}(\cdot,v)\|_{L^2_y} \, .
\end{align}
Let $p := n+1-j \ge 1$. From \eqref{def:Psinj}, when evaluating 
$$\Delta_y \Psi_{n,j} (y,v)= \Delta_y \Big(\langle y \rangle^{-p} \Phi_{n,j}\big (\frac y{\langle y \rangle} , v, K\eps^{\gamma+1}y\big)\Big)$$ we get derivatives of the weight $\langle y \rangle^{-p}$, the $\tilde{r}(y) := y/\langle y \rangle$,  the scaled variable $\tilde{z}(y) := K\eps^{1+\gamma}y$,  and~$\Phi_{n,j}$. In particular, we have  
\begin{align}
    &|\nabla_y (\Phi_{n,j}(\tilde{r}(y),v,\tilde{z}(y)))| \lesssim |\nabla_r \Phi_{n,j}||\nabla_y \tilde r| + \eps^{\gamma}|\nabla_z \Phi_{n,j}| \eps^{-\gamma}|\nabla_y \tilde z|  \lesssim \mathfrak{g}_n\mu_n(v) \big( \langle y \rangle^{-1} + K\eps \big) \,, \\
    &|\Delta_y (\Phi_{n,j}(\tilde{r}(y),v,\tilde{z}(y)))|\lesssim \mathfrak{g}_n\mu_n(v) \big( \langle y \rangle^{-2} + (K\eps)^2 \big)\, .
\end{align}
Therefore, since $|\partial_{y_j}(\langle y\rangle^{-p})|\lesssim p\langle y\rangle^{-(p+1)}$ and $p\leq n+1$, we infer 
\begin{equation}
    |\Delta_y \Psi_{n,j}(y,v)| \lesssim (n+1)^2 \mathfrak{g}_n \mu_n(v) \Big( \frac{1}{\langle y \rangle^{p+2}} + \frac{K^2\eps^2}{\langle y \rangle^{p}} \Big) \, .
\end{equation}
Taking the $L^2_y$ norm we find:
\begin{align}
    \|\Delta_y\Psi_{n,j}(\cdot,v)\|_{L^2_y} &\lesssim (n+1)^2\mathfrak{g}_n\mu_n(v)\bigg(\int_{|y|\leq 2K^{-1}\eps^{-(1+\gamma)}} \Big(\frac{1}{\langle y\rangle^{2(p+2)}} + \frac{K^4\eps^4}{\langle y\rangle^{2p}}\Big) \, dy\bigg)^\frac12 \, .
\end{align}
Since $p \ge 1$, the first integral is uniformly bounded while for the second term we can absorb the logarithmic loss from the integral by taking $\eps$ sufficiently small. Thus, 
\begin{align}
\label{bd:DeltaPsinj}
    \|\Delta_y\Psi_{n,j}(\cdot,v)\|_{L^2_y} \lesssim (n+1)^2\mathfrak{g}_n\mu_n(v) \, .
\end{align}
To handle the low-frequency terms, by H\"older's inequality we get
\begin{align}
    \int_{|\eta|\leq1}|\widehat{\Psi}_{n,j}(\eta,v)| d\eta \lesssim \Big(\int_{|\eta|\leq 1}|\eta|^{-\frac43}d\eta\Big)^\frac34 \Big\||\cdot|\widehat{\Psi}_{n,j}(\cdot,v)\Big\|_{L^4_\eta} \lesssim \|\nabla_y\Psi_{n,j}(\cdot,v)\|_{L^\frac43_y} \, ,
\end{align}
where the last bound follows by the Hausdorff-Young inequality. Arguing as done to prove \eqref{bd:DeltaPsinj}, since~$p\geq 1$ (and $y\in \mathbb{R}^2$) we have 
\begin{align}
    \|\nabla_y\Psi_{n,j}(\cdot,v)\|_{L^\frac43_y} &\lesssim (n+1)\mathfrak{g}_n\mu_n(v)\Big(\int_{|y|\leq 2 K^{-1}\eps^{-(1+\gamma)}}\Big(\frac{1}{\langle y\rangle^{\frac43(p+1)}}+\frac{(K\eps)^\frac43}{\langle y\rangle^{\frac43 p}}\Big)dy\Big)^\frac34\\
    &\lesssim (n+1)\mathfrak{g}_n\mu_n(v)\big(1+\eps^{\frac43} {\eps^{-\frac23(1+\gamma)}} \big)^\frac34.
\end{align} 
Since $0<\gamma< 1$, the last term in the parenthesis is controlled  uniformly in $\eps$. Hence, combining the low and high frequencies bounds and summing over $j$ we get
\begin{align}
    \label{bd:Wienerpf}
    \|g(\cdot,v)\|_{W_x}\lesssim \frac{(n+1)^3}{\eps^{2n+1}K}\mathfrak{g}_n\mu_n(v)\, .
\end{align}

To control the $\mathcal{X}^{1,k}$ norm, we need to sum up \eqref{bd:H1x} and \eqref{bd:Wienerpf} and take the velocity supremum against the weight $\langle v \rangle^k$. Thus, we also have to bound
\begin{equation}
    \sup_{v \in \RR^2} \Big( \langle v \rangle^k \mu_{n}(v) \Big) =  \sup_{v \in \RR^2} \Big( \langle v \rangle^k e^{-\beta_n |v|^2/4} \Big) \, .
\end{equation}
Since we restrict our expansion to $n \le 2M+1$, we know that $$\beta_n = \frac{1}{4} - \frac{n}{32M} \ge  \frac{1}{4} - \frac{1}{8} =\frac{1}{8}\, \cdotp$$ Therefore, the exponential decay is uniformly bounded below by $e^{-|v|^2/32}$, meaning
\begin{equation}
    \sup_{v \in \RR^2} \Big( \langle v \rangle^k e^{-\beta_n |v|^2/4} \Big) \leq \sup_{v \in \RR^2} \Big( \langle v \rangle^k e^{-|v|^2/32} \Big) \leq C_k \, ,
\end{equation}
where $C_k > 0$ is a constant depending only on $k$. Combining this with the spatial sum, we get
\begin{equation}
    \|g\|_{\mathcal{X}^{1,k}} \leq C_k (n+1)^3 \mathfrak{g}_n \frac{1}{K\eps^{2n+1}} \, \virgp
\end{equation}
with an updated constant $C_k>0$.
\end{proof}

Now we study how the operators involved in the Boltzmann equation change the elements in the class in Definition \ref{def:Peps}.
\begin{lemma}\label{lem:algebra}
Let~$n,m\geq 0$  with $n\leq N_M-3$ and $n+m+1\leq N_M-2$. Let $g \in \mathcal{P}_\eps(n)$ with profile bound~$\mathfrak{g}_n$, and $h \in \mathcal{P}_\eps(m)$ with profile bound $\mathfrak{h}_m$. Then there exists a universal constant $C> 0$ (independent of $\eps$, $n$, $m$, and $M$) such that
\begin{enumerate}[label=\roman*)]
    \item \label{prop:transport} $\eps^{-1} v \cdot \nabla_x g \in \mathcal{P}_\eps(n+1)$ with profile bound $\leq C(n+1) M^{3/2} \mathfrak{g}_n, $
    \item \label{prop:linear} $\eps^{-2} L(g) \in \mathcal{P}_\eps(n+1)$ with profile bound $\leq C\sqrt{M}\mathfrak{g}_n,$
    \item \label{prop:nonlinear} $\eps^{-1} \Gamma(g, h) \in \mathcal{P}_\eps(n+m+1)$ with profile bound $\leq C\sqrt{M} \mathfrak{g}_n \mathfrak{h}_m.$
\end{enumerate}
\end{lemma}

\begin{proof}
Recall the identities in \eqref{eq:trivial x rho R} and denote in this proof \begin{equation}
    \label{def:Reps}
    {\sf R}_\eps(x):=\frac{|x|_\eps}{\eps}\, \cdotp
\end{equation}
To prove \ref{prop:transport}, we again consider a typical $(n,j)$-term for an element in the class $\mathcal{P}_\eps(n)$ and we compute that
\begin{align}
    &\eps^{-1} v \cdot \nabla_x \left( \frac{{\sf R}_\eps^j}{\eps^n |x|_\eps^{n+1}} \Phi_{n,j}(r,v,z) \right) \nonumber \\
    &= \frac{1}{\eps^{n+1}|x|_\eps^{n+2}} \bigg[ -(n+1-j)(v\cdot r){\sf R}_\eps^j \Phi_{n,j}   \\
    &\qquad \qquad \qquad + {\sf R}_\eps^j \sum_{i,k} v_i(\delta_{ik} - r_ir_k)\partial_{r_k}\Phi_{n,j} + \eps^{1+\gamma} {\sf R}_\eps^{j+1} (v \cdot \nabla_z \Phi_{n,j}) \bigg] \, .
\end{align}
Since the original profile $\Phi_{n,j}$ possesses bounded derivatives up to order $2M+4-n$, the application of a single spatial derivative leaves the resulting terms with bounded derivatives up to order~$2M+3-n = 2M+4-(n+1)$, which precisely matches the regularity requirement for $\mathcal{P}_\eps(n+1)$ (using also that~$j \leq n$). Each term in the brackets multiplies the profile at most by $\langle v \rangle$. To absorb this linear velocity growth, we slide the weight to the next index $\beta_{n+1} = \beta_n - \frac{1}{32M}$, satisfying
\begin{align}\label{eq:sliding_weight_bound}
    \langle v \rangle \mu_{n}(v)  = \langle v \rangle e^{-\frac{1}{128M} |v|^2} e^{-\beta_{n+1} |v|^2/4} \le \Big( \sup_{v \in \RR^2} \langle v \rangle e^{-\frac{1}{128M} |v|^2} \Big) \mu_{n+1}(v) \, \leq C \sqrt{M}\mu_{n+1}(v).
\end{align}
Furthermore, when evaluating the new profile bound via the factorial-weighted sum \eqref{eq:profile_norm}, applying $\nabla_z$ or $\nabla_r$ increases the derivative multi-index by $1$. To match the new terms to the factorial weight of $\Phi_{n,j}$, we can rewrite the coefficients using $$\frac{1}{\alpha!} = \frac{\alpha_i+1}{(\alpha+e_i)!} \le \frac{2M+4}{(\alpha+e_i)!}\, \cdotp$$ 
Multiplying this by the $\sqrt{M}$ cost from the velocity weight, the new profile is in $\mathcal{P}_\eps(n+1)$ with the profile bound $C(n+1)M^{3/2}\mathfrak{g}_n$.

To prove \ref{prop:linear}, we observe that 
$$\eps^{-2} L \left( \frac{{\sf R}_\eps^j}{\eps^n |x|_\eps^{n+1}} \Phi_{n,j} \right) = \frac{{\sf R}_\eps^{j+1}}{\eps^{n+1} |x|_\eps^{n+2}} L(\Phi_{n,j})\, .$$
Since the spatial derivatives $\nabla_z, \nabla_r$ commute with the linear operator $L$,  it is enough to control the velocity weights. Since $\beta_n\geq 1/8$, by Lemma \ref{lem:Gamma_exponential_bound} we have that
\begin{equation}
| \Phi_{n,j}|\leq \mathfrak{g}_n\mu_n(v)\quad \Longrightarrow\quad |L( \Phi_{n,j})| \leq C \langle v \rangle \mathfrak{g}_n\mu_n(v)\, .
\end{equation}
Applying \eqref{eq:sliding_weight_bound}, we obtain the profile bound $C\sqrt{M} \mathfrak{g}_n$.

Finally, to prove \ref{prop:nonlinear}, we get
\begin{align}
    \eps^{-1} \Gamma \left( \frac{{\sf R}_\eps^{j_1}}{\eps^n |x|_\eps^{n+1}} \Phi_1 \, , \, \frac{{\sf R}_\eps^{j_2}}{\eps^m |x|_\eps^{m+1}} \Phi_2 \right) = \frac{{\sf R}_\eps^{j_1+j_2}}{\eps^{n+m+1} |x|_\eps^{n+m+2}} \Gamma(\Phi_1, \Phi_2) \, .
\end{align}
This produces the correct spatial prefactor for $\mathcal{P}_\eps(n+m+1)$. Since $\Gamma$ acts solely on the velocity variable, applying spatial derivatives to the collision operator requires the Leibniz rule. For any multi-indices $\alpha, \beta$, we have
\begin{equation}
    \nabla_z^\alpha \nabla_r^\beta \Gamma(\Phi_1, \Phi_2) = \sum_{\alpha' \le \alpha} \sum_{\beta' \le \beta} \binom{\alpha}{\alpha'} \binom{\beta}{\beta'} \Gamma\big(\nabla_z^{\alpha'} \nabla_r^{\beta'} \Phi_1, \nabla_z^{\alpha-\alpha'} \nabla_r^{\beta-\beta'} \Phi_2\big) \, . 
\end{equation}
Thus 
\begin{equation}
    \frac{\eps^{\gamma|\alpha|}}{\alpha! \beta!} \big|\nabla_z^\alpha \nabla_r^\beta \Gamma(\Phi_1, \Phi_2)\big| \le \sum_{\substack{\alpha' \le \alpha \\ \beta' \le \beta}} \Big| \Gamma \Big( \frac{\eps^{\gamma|\alpha'|}}{\alpha'! \beta'!} \nabla_z^{\alpha'} \nabla_r^{\beta'} \Phi_1 \, , \, \frac{\eps^{\gamma|\alpha-\alpha'|}}{(\alpha-\alpha')! (\beta-\beta')!} \nabla_z^{\alpha-\alpha'} \nabla_r^{\beta-\beta'} \Phi_2 \Big) \Big| \, .
\end{equation} 
Applying the hard-sphere collision estimates from Lemma \ref{lem:Gamma_exponential_bound}, and summing over all $|\alpha| + |\beta| \le 2M+4-(n+m+1)$, this discrete convolution is  bounded by 
\begin{equation}
    \sum_{\alpha, \beta} \frac{\eps^{\gamma|\alpha|}}{\alpha! \beta!} \big|\nabla_z^\alpha \nabla_r^\beta \Gamma(\Phi_1, \Phi_2)\big| \leq C \langle v \rangle \bigg( \sum_{\alpha', \beta'} \frac{\eps^{\gamma|\alpha'|}}{\alpha'! \beta'!} \big|\nabla_z^{\alpha'} \nabla_r^{\beta'} \Phi_1\big| \bigg) \bigg( \sum_{\alpha'', \beta''} \frac{\eps^{\gamma|\alpha''|}}{\alpha''! \beta''!} \big|\nabla_z^{\alpha''} \nabla_r^{\beta''} \Phi_2\big| \bigg) \, .\end{equation} 
Evaluating the sums, the individual profiles are bounded by $\mathfrak{g}_n\mu_{n}(v)$ and $\mathfrak{h}_m\mu_{m}(v)$. Since $\beta_{n+m} \le \min(\beta_n, \beta_m)$, both Gaussian factors are trivially bounded by $\mu_{n+m}(v)$. Therefore, the nonlinear operator yields a pointwise decay of $C \mathfrak{g}_n \mathfrak{h}_m \langle v \rangle \mu_{n+m}(v)$. Sliding the weight down by one step to $\beta_{n+m+1}$ and invoking \eqref{eq:sliding_weight_bound},  yields the profile bound $C\sqrt{M} \mathfrak{g}_n \mathfrak{h}_m$.\end{proof}

\subsection{Iterative construction of the approximate solution}

We are now ready to  construct the approximate solution up to an arbitrary order $M \ge 1$. We define the $M$-th order approximate solution as
\begin{equation}
\label{def:fappM}
    f_{\rm app}^{\eps(M)}(t,x,v) := \sum_{n=0}^M t^n f_n^\eps(x,v) \, ,
\end{equation}
where $f_0^\eps(x,v) := f_{\rm in}^\eps(x,v) \in \mathcal{P}_\eps(0)$ (see Remark~\ref{def:profilef0})  is defined in~(\ref{def:rhofin}) and we seek to define iteratively each $f_n^\eps(x,v)$  in~$ \mathcal{P}_\eps(n)$.
To measure the accuracy of this approximation, we define the error at order $M$ as the remainder obtained by plugging $f_{\rm app}^{\eps(M)}$ directly into the  Boltzmann equation \eqref{eq:feps}
\begin{equation}
\label{def:ErrorM}
    \cE^{\eps(M)}(t) := \de_t f_{\rm app}^{\eps(M)} + \eps^{-1} v \cdot \nabla_x f_{\rm app}^{\eps(M)} - \eps^{-2} L(f_{\rm app}^{\eps(M)}) - \eps^{-1} \Gamma(f_{\rm app}^{\eps(M)}, f_{\rm app}^{\eps(M)}) \, .
\end{equation}
By expanding \eqref{def:ErrorM} and grouping the terms by powers of $t$, we enforce that all terms up to order $t^{M-1}$ vanish. This yields a direct recursive definition for the functions $f_n^\eps$, mimicking the informal expansion of Section \ref{sec:naive}, but now grounded in our  functional framework. The main goal is to prove the following.
\begin{lemma}\label{lem:recurrence_bounds}
For any $1 \le n \le M$, let the terms $f_n^\eps$ be defined recursively by~$f_0^\eps = f^\eps_{\rm in}$ as in~\eqref{def:rhofin}
and
\begin{equation}\label{eq:fn_recurrence}
 \forall n \geq 1 \, , \quad    f_n^\eps := \frac{1}{n} \bigg[ \eps^{-2}L(f_{n-1}^\eps) - \eps^{-1}v\cdot \nabla_x f_{n-1}^\eps + \eps^{-1}\sum_{k=0}^{n-1}\Gamma(f_k^\eps, f_{n-1-k}^\eps) \bigg] \, .
\end{equation}
Then $f_n^\eps \in \mathcal{P}_\eps(n)$. Furthermore, if $M\leq C_0|\ln(\eps)|$, there exists $\eps_0(C_0)>0$ and~$C_* > 0$ (independent of $\eps_0$, $K$, $n$, and $M$) such that for all $0<\eps<\eps_0$ the profile bound $\mathfrak{f}_n$ associated with $f_n^\eps$ satisfies:
\begin{equation}\label{eq:profile_geom_bound}
 \forall\, 
 0\leq n \leq M \, , \quad    \mathfrak{f}_n \leq  (C_* M^{3/2})^n \,\mathfrak{f}_0\, .
\end{equation}
Consequently, for any fixed $k \ge 0$, there exists a constant $C_k > 0$ such that 
\begin{equation}\label{eq:fn_X_bound}
    \| f_n^\eps \|_{\mathcal{X}^{1,k}} \leq C_k (n+1)^3 (C_* M^{3/2})^n \frac{1}{K \eps^{2n+1}} \, \cdotp
\end{equation}
Finally, the error term defined in~\eqref{def:ErrorM}
evaluated at time $t$ satisfies the bound
\begin{equation}\label{eq:error_bound}
    \| \cE^{\eps(M)}(t) \|_{\mathcal{X}^{1,k}} \leq \frac{C_k}{K \eps^3} \sum_{j=M}^{2M} (j+2)^4 (C_* M^{3/2})^{j+1} \left(\frac{t}{\eps^2}\right)^j \, .
\end{equation}
\end{lemma}

\begin{proof}
We proceed by induction on $n$. For the base case $n=0$, the initial data $f_0^\eps \in \mathcal{P}_\eps(0)$ possesses a smooth, bounded profile $\Phi_0$ as given in \eqref{def:rhofin}. Its profile bound $\mathfrak{f}_0 > 0$ is a universal constant determined in \eqref{eq:boundf0}.

Assume that for all $m \le n-1$, $f_m^\eps$ belongs to $\mathcal{P}_\eps(m)$ with profile bounds $\mathfrak{f}_m$ satisfying \eqref{eq:profile_geom_bound}. By Lemma \ref{lem:algebra}, the operators acting on these terms yield functions belonging to $\mathcal{P}_\eps(n)$. Specifically, the terms $\eps^{-2}L(f_{n-1}^\eps)$ and $\eps^{-1}v\cdot \nabla_x f_{n-1}^\eps$ produce profiles bounded by $C n M^{3/2} \mathfrak{f}_{n-1}$. Similarly, each nonlinear interaction $\eps^{-1}\Gamma(f_k^\eps, f_{n-1-k}^\eps)$ produces a profile 
in~$\mathcal{P}_\eps(n)$
bounded by $C \sqrt{M} \mathfrak{f}_k \mathfrak{f}_{n-1-k}$. 
Using definition \eqref{eq:fn_recurrence}, we get the profile bound $\mathfrak{f}_n$
\begin{equation}
    \mathfrak{f}_n \leq \frac{1}{n} \bigg[ C n M^{3/2} \mathfrak{f}_{n-1} + C \sqrt{M} \sum_{k=0}^{n-1} \mathfrak{f}_k \mathfrak{f}_{n-1-k} \bigg] \leq C M^{3/2}\Big( \mathfrak{f}_{n-1} +  \sum_{k=0}^{n-1} \mathfrak{f}_k \mathfrak{f}_{n-1-k}\Big) \, .
\end{equation}
To control this recursive inequality, we bound $\mathfrak{f}_n$ using a scaled sequence of Catalan numbers $\mathfrak{C}_n$, which are defined as $\mathfrak{C}_n \displaystyle = \sum_{k=0}^{n-1} \mathfrak{C}_k \mathfrak{C}_{n-1-k}$ with $\mathfrak{C}_0 = 1$. We claim that
\begin{equation}\label{eq:claimfn} \forall n \geq 0 \, , \quad \mathfrak{f}_n \leq \mathfrak{f}_0 A^n \mathfrak{C}_n
\end{equation}
 for a suitable constant $A > 0$. The base case $n=0$ holds trivially. For the inductive step, substituting the ansatz yields
\begin{equation}
    \mathfrak{f}_n \leq C M^{3/2} \bigg( \mathfrak{f}_0 A^{n-1} \mathfrak{C}_{n-1} + \sum_{k=0}^{n-1} (\mathfrak{f}_0A^k \mathfrak{C}_k)(\mathfrak{f}_0 A^{n-1-k} \mathfrak{C}_{n-1-k}) \bigg) \, .
\end{equation}
Using the Catalan identity, the sum evaluates to $\mathfrak{f}_0^2 A^{n-1} \mathfrak{C}_n$. Thus to prove~\eqref{eq:claimfn}, we require
\begin{equation}
    C M^{3/2} \big( \mathfrak{f}_0 A^{n-1} \mathfrak{C}_{n-1} + \mathfrak{f}_0^2 A^{n-1} \mathfrak{C}_n \big) \leq \mathfrak{f}_0 A^n \mathfrak{C}_n \, .
\end{equation}
Dividing by $\mathfrak{f}_0A^{n-1} \mathfrak{C}_n$ and using the fact that $\mathfrak{C}_{n-1} / \mathfrak{C}_n \le 1$ for all $n \ge 1$, the condition reduces to $C M^{3/2} (1 + \mathfrak{f}_0) \leq A$. Therefore, setting $A = C M^{3/2} (1 + \mathfrak{f}_0)$, the induction is closed and the claim~(\ref{eq:claimfn}) is proved. 

Using the standard upper bound $\mathfrak{C}_n \le 4^n$, we obtain $\mathfrak{f}_n \le \mathfrak{f}_0 (4A)^n$. By defining the constant $C_* = 4C(1+\mathfrak{f}_0)\max\{1, \mathfrak{f}_0\}$, we arrive at the simplified bound \eqref{eq:profile_geom_bound}.

The $\mathcal{X}^{1,k}$ bound \eqref{eq:fn_X_bound} follows by applying Lemma \ref{lem:analytic_bounds} to $f_n^\eps \in \mathcal{P}_\eps(n)$ with the   profile bound \eqref{eq:profile_geom_bound}.

To bound the error $\cE^{\eps(M)}(t)$, we note that all terms of order $t^j$ for $j \le M-1$ cancel by the definition of $f_n^\eps$. The remaining terms are:
\begin{equation}
    \cE^{\eps(M)}(t) = t^M \big( \eps^{-1} v \cdot \nabla_x f_M^\eps - \eps^{-2} L(f_M^\eps) \big) - \eps^{-1} \sum_{j=M}^{2M} t^j \sum_{\substack{p+q=j \\ 0 \le p,q \le M}} \Gamma(f_p^\eps, f_q^\eps) \, .
\end{equation}
By Lemma~\ref{lem:algebra},  linear operators on $f_M^\eps$ generate a profile in~$\mathcal{P}_\eps(M+1)$ with bound $$C(M+1)M^{3/2}(C_*M^{3/2})^M\, .$$ For any $M \le j \le 2M$, the convolution sum has at most $j+1$ terms, each yielding a profile in  the set~$\mathcal{P}_\eps(j+1)$ bounded by $C\sqrt{M}(C_*M^{3/2})^j$. 
Summing these, the total profile bound at order $t^j$ is bounded by~$ (j+1)(C_*M^{3/2})^{j+1}$. 
Applying Lemma \ref{lem:analytic_bounds}, the $\mathcal{X}^{1,k}$ norm of the $t^j$ contribution incurs an additional factor $(j+2)^3 (K\eps^{2j+3})^{-1}$. Thus, the norm of the order $t^j$ term is bounded by
\begin{equation}
  (j+2)^4 \frac{1}{K\eps^3} \left(\frac{t}{\eps^2}\right)^j (C_*M^{3/2})^{j+1}\, .
\end{equation}
Summing over $j \in \{M, \dots, 2M\}$ yields the final pointwise error bound \eqref{eq:error_bound}.
\end{proof}

The bound \eqref{eq:fn_X_bound} provides the following estimate on $f_{\rm app}^{\eps(M)}$ defined in \eqref{def:fappM} and $\mathcal{E}^{\eps(M)}$ in \eqref{def:ErrorM} for the initial time-layer under consideration.

\begin{corollary}\label{cor:bound on the approximation}
With the notation of Lemma \ref{lem:recurrence_bounds}, if $\delta$ is small enough then for all~$t \leq \delta \displaystyle \frac{\eps^2}{C_* M^{3/2}}$ and all~$k \geq 0$ there holds 
\begin{align}
\|  f_{\rm app}^{\eps(M)}(t)\|_{\mathcal{X}^{1,k}} &\lesssim  \frac1{K \eps}\quad \text{if } M \geq 1 \, , 
 \\
    \| \cE^{\eps(M)}(t) \|_{\mathcal{X}^{1,k}} &\lesssim \eps^{N} \quad \text{if } M=(N+4)\left\lceil\frac{|\ln(\eps)|}{|\ln(\delta)|}\right\rceil \text{ and } N\geq 1\, .
\end{align}
\end{corollary}
\begin{proof}
By \eqref{def:fappM} and \eqref{eq:fn_X_bound}, one has
 $$
 \begin{aligned}
   \|  f_{\rm app}^{\eps(M)}(t)\|_{\mathcal{X}^{1,k}} &\leq  \sum_{n=0}^M t^n     \| f_n^\eps(x,v)\|_{\mathcal{X}^{1,k}} \\
    &\lesssim \frac1{K \eps} \sum_{n=0}^M  (n+1)^3 \Big(\frac{C_* t M^{3/2}}{\eps^2}\Big)^n  \, . \end{aligned}
 $$
 The result follows since the sum is clearly uniformly bounded for all $t \displaystyle \leq \delta \frac{\eps^2}{C_* M^{3/2}}$ and $\delta$ small enough. 
  
 Similarly, appealing to the bound \eqref{eq:error_bound}, we get 
\begin{equation}
    \| \cE^{\eps(M)}(t) \|_{\mathcal{X}^{1,k}} \leq \frac{C_k C_* M^{3/2}}{K \eps^3} \sum_{j=M}^{2M} (j+2)^4 \delta^j \lesssim  \frac{M^{11/2}}{K\eps^3} \delta^M \, .
\end{equation}
By taking $M$ as in the statement, it is not hard to verify that indeed the error is smaller than~$  \eps^{N}$ for $\eps$ small enough. 
\end{proof}

\subsection{Macroscopic bounds for the approximate solution}

With the sequence $f_n^\eps$ rigorously constructed, we can now compute its macroscopic moments to verify the short-time instability bounds for the approximate solution. Recall that  $${\bf P}f = \Big(\rho[f] + u[f]\cdot v + \theta[f](\frac{|v|^2}{2}-1)\Big)\sqrt{\mu}\,,$$
where $\rho,u,\theta$ are the moments   defined in \eqref{eq:moments}. 

To localize our analysis near the vortex core where the cut-off $\chi(\eps^\gamma |x|)$ is identically $1$, we restrict ourselves to a region $|x| \sim  \eps$ and we exploit that $ \mathcal{X}^{1,k}\hookrightarrow L^\infty_x$. Moreover, in this region, the initial velocity is exactly $\displaystyle u_{\rm in}^\eps(x) = \frac{1}{2\pi} \frac{x^\perp}{|x|_\eps^2}\cdotp $

\begin{lemma}\label{lem:macro_bounds}
Let $f_{\rm app}^{\eps(M)}$ be the approximate solution up to order $3\leq M\leq C_0|\ln(\eps)|$, and let
$u[f^\eps_{\rm LO}]$ be the velocity of the regularized kinetic Lamb--Oseen vortex.
Let $C_*$ be the constant in Lemma~\ref{lem:recurrence_bounds}. 
Then there exist constants $0<a_0<b_0<\infty$, $c>0$, and $\delta>0$, independent of
$\eps$ and $M$, such that, setting
\[
    T_\eps:=\delta \frac{\eps^2}{C_*M^6} \,  \virgp
    \qquad
    A_\eps:=\big\{x\in\mathbb R^2\, / \,  a_0\eps\le |x|\le b_0\eps\big\}\, ,
\]
the macroscopic moments of $f_{\rm app}^{\eps(M)}$ satisfy, for all
$x\in A_\eps$ and $0<t\le T_\eps$,
\begin{align}
      \big|u_{\rm tan}[f_{\rm app}^{\eps(M)}](t,x)-u_{\rm tan}[f^{\eps}_{\rm LO}](t,x)\big|\geq c \frac{t}{\eps^3}\virgp \qquad
       \big|u_{\rm rad}[f_{\rm app}^{\eps(M)}](t,x)\big|
      &\geq c \frac{t^2}{\eps^5}\,  \virgp
 \label{eq:app_bound_u} \\
      \big|\nabla \cdot u[f_{\rm app}^{\eps(M)}](t,x)\big|
      &\geq c \frac{t^2}{\eps^6}\,  \virgp
 \label{eq:app_bound_div} \\
      \big|\rho[f_{\rm app}^{\eps(M)}](t,x)\big| +    \big|\theta[f_{\rm app}^{\eps(M)}](t,x)\big|
      &\geq c \frac{t^3}{\eps^7}\,  \cdotp
 \label{eq:app_bound_rho}
\end{align}
\end{lemma}

\begin{proof}
By linearity of the projection operator, we can study the macroscopic moments of the approximate solution order by order. From now on, since $x \in A_\eps$, the cut-off for large~$|x|$ is no longer present. Before proceeding, we extract the precise pointwise bounds near the vortex core from Lemma~\ref{lem:recurrence_bounds}.   Let
\begin{equation}
    \mathfrak m[g]:=(\rho[g],u[g],\theta[g])\, .
\end{equation}
Since~$W_x\hookrightarrow L^\infty_x$ and the $\mathcal{X}^{1,k}$ norm controls the Wiener norm after integrating the Gaussian against the macroscopic moments, we deduce that 
\begin{align}
\label{bd:macrofn}
   \big |\mathfrak m[t^n f_n^\eps]\big|(x) \lesssim \frac{1}{\eps} (n+1)^3\left(\frac{tC_* M^{3/2}}{\eps^2}\right)^n  \, .
\end{align}
   Moreover, we shall also use the following elementary pointwise consequence of the definition of
$\mathcal P_\eps(n)$. If a function~$g\in\mathcal P_\eps(n)$ has profile bound $\mathfrak g_n$, then, for every $x\in A_\eps$ and $|\alpha|\le 1$ we have
\begin{equation}\label{eq:local_pointwise_Pn}
    \big|\nabla_x^\alpha \mathfrak m[g](x)\big|
    \le C\,(n+1)^2\,\mathfrak g_n\,
    \eps^{-2n-1-|\alpha|}\, ,
    \qquad x\in A_\eps \,,
\end{equation}
where the constant $C>0$ only depends on $a_0,b_0,K$. Indeed, on $A_\eps$ one has $|x|_\eps\sim\eps$, and each term in the
definition of $\mathcal P_\eps(n)$ is therefore of size $\eps^{-2n-1}$.
One spatial derivative can fall either on the explicit powers of $|x|_\eps$, on
$r=x/|x|_\eps$, or on $z=\eps^\gamma x$. The first two contributions cost at most
one additional factor $\eps^{-1}$, while the $z$-derivative is lower order in the core. Since the Gaussian velocity weights are integrable against the moment polynomials, the bound above follows.

Consequently, using $\mathfrak f_n\le (C_*M^{3/2})^n$, we can bound the tail of the expansion. Namely, for $|\alpha|\le1$ and $x\in A_\eps$ we have
\begin{equation}\label{eq:macro_tail_derivative0}
    \left|
    \nabla_x^\alpha
    \mathfrak m\!\left[\sum_{n=N+1}^M t^n f_n^\eps\right](x)
    \right|
    \le
    C
    \eps^{-1-|\alpha|}
    \left(\frac{C_*M^{3/2}t}{\eps^2}\right)^{N+1}\sum_{j=0}^{M-N-1} (N+2+j)^3 \left( \frac{t C_* M^{3/2}}{\eps^2} \right)^j.
\end{equation}
Thus, provided $t\le \delta\eps^2/(C_*M^{3/2})$ and $\delta>0$ is sufficiently small we infer 
\begin{equation}\label{eq:macro_tail_derivative}
    \left|
    \nabla_x^\alpha
    \mathfrak m\!\left[\sum_{n=N+1}^M t^n f_n^\eps\right](x)
    \right|
    \le
    C
    \eps^{-1-|\alpha|}N^3
    \left(\frac{C_*M^{3/2}t}{\eps^2}\right)^{N+1}\, .
\end{equation}
This formulation allows us to effectively separate the exact macroscopic moments from the higher-order corrections.

Moreover, in the following computations,  we use the symmetrized notation
\begin{equation}
     \Gamma_{\rm sym}(f_k,f_j):=\begin{cases}
         \Gamma(f_k,f_j)+\Gamma(f_j,f_k)\qquad &\text{ if } k\neq j\\
         \Gamma(f_k,f_k) \qquad &\text{ if } k=j.
     \end{cases}
\end{equation}   
Accordingly, the recurrence in  \eqref{eq:fn_recurrence} can be rewritten as
\begin{align}
     \sum_{k=0}^{n-1}\Gamma(f_k^\eps,f_{n-1-k}^\eps)= \sum_{k=0}^{\lfloor\frac{n-1}{2}\rfloor}\Gamma_{\rm sym}(f_k^\eps,f_{n-1-k}^\eps).
\end{align}   
This is useful because the collision invariants imply that
\begin{equation}
    {\bf P}\Gamma_{\rm sym}(f_k,f_{n-1-k})=0.
\end{equation}
We now proceed by checking our expansion order by order.

\medskip \noindent $\diamond$ \textbf{Order $n=0$.}
We have $f_0^\eps(x,v) = u_{\rm in}^\eps(x) \cdot v \sqrt{\mu}$. By the parity of the Gaussian, its moments are $\rho_0 = \theta_0 = 0$ and $u_0(x) = u_{\rm in}^\eps(x)$, which is a purely tangential vector field. 

\medskip \noindent $\diamond$ \textbf{Order $n=1$.}
Using \eqref{eq:fn_recurrence} and recalling $L(f_0^\eps) = 0$, we have $$f_1^\eps = \eps^{-1} \big( -v \cdot \nabla_x f_0^\eps + \Gamma(f_0^\eps, f_0^\eps) \big).$$ Since ${\bf P} \Gamma(f,f) = 0$, it is enough to compute
\begin{equation}
    {\bf P}(v \cdot \nabla_x f_0^\eps) = {\bf P} \big( v \cdot \nabla_x (u_{\rm in}^\eps \cdot v) \sqrt{\mu} \big) = \sum_{j,k} \partial_j u_{\rm in, k}^\eps {\bf P}(v_j v_k \sqrt{\mu}) \, .
\end{equation}
The velocity component of its projection contains third Gaussian moments and therefore
vanishes. Thus,~$u[f_1^\eps] = 0$. For the density, we must have $j=k$ (since otherwise we use $\displaystyle \int v_j v_k \mu = 0$), yielding
\begin{align}
 \rho[f_1^\eps] = -\eps^{-1} \nabla \cdot u_{\rm in}^\eps = 0 \, ,
\end{align}
where we used that the regularized initial data $u_{\rm in}^\eps$ is divergence-free. Similarly, $\theta[f_1^\eps] = 0$. Therefore, the first-order correction is purely microscopic, ${\bf P} f_1^\eps = 0$, and 
\begin{align}
\label{eq:f0+tf1}
    u[f_0^\eps + t f_1^\eps] = u[f_0^\eps] = u_{\rm in}^\eps \, .
\end{align}
We now evaluate the difference with the regularized kinetic Lamb--Oseen vortex. The Taylor expansion of the heat semigroup \eqref{eq:TayloruLO} evaluated at $|x| \sim \eps$ is:
\begin{align}
    u[f^\eps_{\rm LO}](t,x) = u_{\rm in}^\eps(x) + \nu t \Delta_x u_{\rm in}^{\eps}(x) + \mathcal{O}\big((\nu t)^2\|\Delta^2 u_{\rm in}^\eps\|_{L^\infty(\mathbb{R}^2)}\big) \, ,
\end{align}
where \begin{equation}\label{eq:laplacian_uin}
    |\Delta u_{\rm in}^\eps(x)|
    =
    \frac{4K^2\eps^2}{\pi}
    \frac{|x^\perp|}{|x|_\eps^6}\sim \frac{1}{\eps^{3}} \qquad \text{ for } x\in A_\eps\, .
\end{equation}
Since  $$\|\Delta^2u_{\rm in}^\eps\|_{L^\infty(\mathbb{R}^2)}
\lesssim \eps^{-5},$$ 
using \eqref{eq:f0+tf1} and bounding the approximation tail via \eqref{eq:macro_tail_derivative} with $N=1$, we deduce:
\begin{align}
    \big| u[f_{\rm app}^{\eps(M)}] - u[f_{\rm LO}^\eps] \big| 
    &\geq \nu t \big|\Delta_x u_{\rm in}^\eps\big| - \mathcal{O}\Big(\frac{(\nu t)^2}{\eps^5}\Big) - \Big| u\Big[ \sum_{n=2}^M t^n f_n^\eps \Big] \Big|  \\
    &\approx \frac{t}{\eps^3} - \mathcal{O}\Big( \frac{t^2 M^3}{\eps^5} \Big) \, .
\end{align}
Having that $t \leq T_\eps \ll \eps^2 / M^3$, we can safely absorb the error and therefore we get the desired lower bound for the tangential velocity difference. Indeed, the velocity $u[f_0^\eps+tf_1^\eps]$ is exactly tangential and thus we have proved the first part of \eqref{eq:app_bound_u}.

\medskip \noindent $\diamond$ \textbf{Order $n=2$.}
By \eqref{eq:fn_recurrence}, $$f_2^\eps = \frac{1}{2} \eps^{-1} \big( \eps^{-1}L(f_1^\eps) - v \cdot \nabla_x f_1^\eps + \Gamma_{\rm sym}(f_0^\eps, f_1^\eps) \big).$$ Since ${\bf P} L = {\bf P} \Gamma_{\rm sym} = 0$, the macroscopic moments are  driven by $-\frac{1}{2\eps} {\bf P} (v \cdot \nabla_x f_1^\eps)$.

Because $f_1^\eps$ has even parity in $v$, multiplying by $v$ makes it odd. Thus, the density and temperature moments vanish $\rho[f_2^{\eps}] = \theta[f_2^{\eps}] = 0$. For the velocity moment, we have
\begin{align}
    u[f_2^\eps] = -\frac{1}{2\eps} \int_{\mathbb{R}^2} v (v \cdot \nabla_x f_1^\eps)\sqrt{\mu}(v) dv = -\frac{1}{2\eps} \nabla_x \cdot \mathcal{T}_1^\eps \, ,
\end{align}
where $$\mathcal{T}_1^\eps := \int_{\mathbb{R}^2} (v\otimes v) f_1^\eps \sqrt{\mu}(v) dv$$ is the stress tensor generated at the previous order. Recalling that $f_1^\eps = \eps^{-1} \big( -v \cdot \nabla_x f_0^\eps + \Gamma(f_0^\eps, f_0^\eps) \big)$, we can compute the contributions of the linear and nonlinear parts directly. 

For the linear part, since $f_0^\eps = u_{\rm in}^\eps \cdot v \sqrt{\mu}$, we are evaluating a fourth-order moment of the Gaussian, which gives us
\begin{align}
    \frac{1}{\eps} \int_{\mathbb{R}^2} v_j v_k (-v \cdot \nabla_x f_0^\eps) \sqrt{\mu}(v) dv &= -\frac{1}{\eps}\sum_{\ell,m}\partial_l u_{{\rm in},m}^\eps\int_{\mathbb{R}^2}    v_j v_k v_\ell v_m \mu \, dv\\
    &=-\frac{1}{\eps} \big( \partial_j u_{{\rm in}, k}^\eps + \partial_k u_{{\rm in}, j}^\eps  + \delta_{jk} \nabla \cdot u_{\rm in}^\eps)\big) \, .
\end{align}
Since $\nabla \cdot u_{\rm in}^\eps = 0$, the divergence of the tensor above yields exactly $-\eps^{-1} \Delta u_{\rm in}^\eps.$

For the nonlinear part, we must evaluate the stress tensor 
\begin{align}
\mathcal{B}_{jk}^\eps(u_{\rm in}^\eps) := \frac{1}{\eps} \int_{\mathbb{R}^2} v_j v_k \Gamma( u_{\rm in}^\eps \cdot v \sqrt{\mu}, u_{\rm in}^\eps \cdot v \sqrt{\mu}) \sqrt{\mu} dv.
\end{align}
The tensor $\mathcal{B}^\eps(u_{\rm in}^\eps)$ depends quadratically on $u_{\rm in}^\eps$ and is  identified with a symmetric $2\times 2$ matrix of the form 
\begin{align}
\begin{pmatrix}
a_1 (u_{\rm in,1}^\eps)^2+b_1 (u_{\rm in,2}^\eps)^2 & c\, u_{\rm in,1}^\eps u_{\rm in,2}^\eps\\
c \,u_{\rm in,1}^\eps u_{\rm in,2}^\eps & a_2 (u_{\rm in,1}^\eps)^2+b_2 (u_{\rm in,2}^\eps)^2
\end{pmatrix}
\end{align}
for suitable constants $a_i,b_i,c$. Then, since the collision operator is invariant by rotation and reflection, for any orthogonal matrix $Q\in O(2)$ we know that 
\begin{align}
    \mathcal{B}^\eps(Qu)=Q\mathcal{B}^\eps(u)Q^T.
\end{align}
Thus,  for a fixed $x$, we can choose a reference frame where $u_2=0$. In that case it is trivial to deduce that $\mathcal{B}^\eps(u^\eps_{\rm in})$ is a linear combination of $u_{\rm in}^\eps\otimes u_{\rm in}^\eps$ and $|u_{\rm in}^\eps|^2 \mathbb{I}$. This is clearly independent on the particular reference frame and therefore there exist two constants $\mathsf{c}_1, \mathsf{c}_2$ (depending only on the collisions) such that
\begin{equation}
\label{def:Buin}
    \mathcal{B}^\eps(u_{\rm in}^\eps) = \frac{1}{\eps} \Big( \mathsf{c}_1(u_{\rm in}^\eps \otimes u_{\rm in}^\eps) + \mathsf{c}_2 |u_{\rm in}^\eps|^2 \mathbb{I} \Big) \, .
\end{equation}
Moreover, by energy conservation we also know that $\mathrm{tr}(\mathcal{B}^\eps(u))=0$ for any $u$, meaning that 
$$\mathsf{c}_1+2\mathsf{c}_2=0.$$
For the hard-sphere collision kernel, we can guarantee that ${\sf c}_1> 0$ as we show in Appendix~\ref{sct:collision constant}, meaning that the contribution from this stress tensor is not trivial.
Taking the divergence, and exploiting again that $u_{\rm in}^\eps$ is divergence-free, we obtain
\begin{equation}
    \nabla \cdot (\mathcal{B}^\eps(u_{\rm in}^\eps)) = \frac{\mathsf{c}_1}{\eps} (u_{\rm in}^\eps \cdot \nabla) u_{\rm in}^\eps + \frac{\mathsf{c}_2}{\eps} \nabla |u_{\rm in}^\eps|^2 \, .
\end{equation}
Combining these components, the divergence of the stress tensor $\mathcal{T}_1^\eps$ is then given by 
\begin{equation}
    u[f_2^\eps]=-\frac{1}{2\eps}\nabla \cdot \mathcal{T}_1^\eps = \frac{1}{2\eps^2}\Big( \Delta u_{\rm in}^\eps - \mathsf{c}_1 (u_{\rm in}^\eps \cdot \nabla) u_{\rm in}^\eps - \mathsf{c}_2\nabla |u_{\rm in}^\eps|^2 \Big)\, .
\end{equation}
Here we know that  $\Delta u_{\rm in}^\eps$ is purely tangential. However, 
\begin{align}
(u_{\rm in}^\eps \cdot \nabla)u_{\rm in}^\eps = -\frac{1}{4\pi^2} \frac{x}{|x|_\eps^4}\, ,\,  \qquad \nabla |u_{\rm in}^\eps|^2 = \frac{x}{4\pi^2}  \frac{2|x|_\eps^2 - 4|x|^2}{|x|_\eps^6}
\end{align}
are both radial. Thus, collecting the radial contributions and using that $\mathsf{c}_2=-\mathsf{c}_1/2$ we have
\begin{equation}
    u_{\rm rad}[f_2^\eps](x) = \frac{1}{8\pi^2\eps^2} \left( \frac{\mathsf{c}_1}{|x|_\eps^4} - \mathsf{c}_2 \frac{2|x|_\eps^2 - 4|x|^2}{|x|_\eps^6} \right) |x| \, =\frac{\mathsf{c}_1K^2}{4\pi^2}\frac{|x|^2}{|x|_\eps^6}.
\end{equation}
Hence, there exists $a_0,b_0$  such that the coefficient above does not vanish for all $x\in A_{\eps}$. Consequently,
\begin{equation}
    |u_{\rm rad}[f_2^\eps](x)| \approx \frac{C_{\rm rad}}{\eps^5}  \, , \qquad \text{ for } x\in A_\eps,
\end{equation}
where $C_{\rm rad}>0$ is a constant depending on $\mathsf{c}_1$, and $K$. Since $u_{\rm rad}[f_0^\eps+tf_1^\eps+t^2f_2^\eps]=t^2u_{\rm rad}[f_2^\eps]$, we can argue as before and use \eqref{eq:macro_tail_derivative} with $N=2$ to conclude that 
\begin{equation}
    |u_{\rm rad}[f_{\rm app}^{\eps(M)}](x)| \approx \frac{t^2}{\eps^5} - \mathcal{O}\Big( \frac{t^3 M^{9/2}}{\eps^7} \Big) \, , \qquad \text{for } x\in A_\eps
\end{equation}
Since $t \leq T_\eps \ll \eps^2 / M^{9/2}$, the error is  dominated by the main term and the proof of \eqref{eq:app_bound_u} is over.

Furthermore, let us compute the error in the divergence-free condition. Since $\nabla_x\cdot \Delta_x u_{\rm in}^\eps=0$, we only need to compute the divergence of the radial components found above. Applying the divergence operator, and observing that each spatial derivative increases the singularity degree by $1$, for $x\in A_\eps$ we get
\begin{equation}
  \big  |\nabla \cdot u[f_2^\eps](x)\big| = \frac{1}{8\pi^2\eps^2} \left| -\mathsf{c}_1 \nabla \cdot \Big( \frac{x}{|x|_\eps^4} \Big) - \mathsf{c}_2 \nabla \cdot \Big( x\frac{2|x|_\eps^2 - 4|x|^2}{|x|_\eps^6} \Big) \right| \approx \eps^{-6} \, .
\end{equation}
Multiplying by $t^2$ and explicitly bounding the tail as before, we get
\begin{equation}
  \big  |\nabla \cdot u[f_{\rm app}^{\eps(M)}](x)\big| \approx \frac{t^2}{\eps^6} - \mathcal{O}\Big( \frac{t^3 M^{9/2}}{\eps^8} \Big) \, ,  \qquad \text{for } x\in A_\eps,
\end{equation}
thus proving \eqref{eq:app_bound_div} in the time-scale under consideration.

\medskip \noindent $\diamond$ \textbf{Order $n=3$.}
By \eqref{eq:fn_recurrence}, $f_3^\eps$ is given by
\begin{equation}
    f_3^\eps = \frac{1}{3} \bigg[ \eps^{-2}L(f_2^\eps) - \eps^{-1}v\cdot \nabla_x f_2^\eps + \eps^{-1} \big( \Gamma_{\rm sym}(f_0^\eps, f_2^\eps) + \Gamma(f_1^\eps, f_1^\eps) \big) \bigg] \, .
\end{equation}
We are interested in the macroscopic density $\rho[f_3^\eps]$ and temperature $\theta[f_3^\eps]$. These moments are driven by the transport of the second-order correction and one has that
\begin{align}
    \rho[f_3^\eps] = -\frac{1}{3\eps} \nabla_x \cdot \int_{\RR^2} v f_2^\eps \sqrt{\mu} dv = -\frac{1}{3\eps} \nabla_x \cdot u[f_2^\eps] \, , \\
    \theta[f_3^\eps] = -\frac{1}{3\eps} \nabla_x \cdot \int_{\RR^2} v \Big(\frac{|v|^2}{2} - 1\Big) f_2^\eps \sqrt{\mu} dv =: -\frac{1}{3\eps} \nabla_x \cdot q[f_2^\eps] \, ,
\end{align}
where $q[f_2^\eps]$ represents the heat flux vector generated at order 2. 

We computed $\nabla_x \cdot u[f_2^\eps]$ in the previous step and found it scales as $\eps^{-6}$, and therefore $\rho[f_3^\eps]\approx \eps^{-7}$. Analogously, one can compute the heat flux vector driving the temperature  appealing to higher-order moments of the Gaussian and introducing a new stress tensor arising from the nonlinearity. One can perform the detailed computations for this part to deduce that also $|\theta[f_3^\eps]|\approx \eps^{-7}$, but since clearly $|\theta[f_3^\eps]|\lesssim \eps^{-7}$ we only state that for $|x|\sim \eps$
\begin{equation}
    |\rho[f_3^\eps]| + |\theta[f_3^\eps]| \approx \eps^{-7} \, .
\end{equation}
Applying \eqref{eq:macro_tail_derivative} for $N=3$, we obtain
\begin{equation}
    |\rho[f_{\rm app}^{\eps(M)}]| + |\theta[f_{\rm app}^{\eps(M)}]| \approx \frac{t^3}{\eps^7} - \mathcal{O}\Big( \frac{t^4 M^6}{\eps^9} \Big) \, .
\end{equation}
Once again, within the required time-scale $t \leq T_\eps \ll \eps^2 / M^6$ we recover the desired  bound \eqref{eq:app_bound_rho} and thus  conclude the proof.
\end{proof}
\section{Correcting the approximate solution}\label{eq:correction}
 In this section we prove that the function~$f_{\rm app}^{\eps(M)}$ constructed in~(\ref{def:fappM})  is indeed an approximate solution to the rescaled Boltzmann equation~\eqref{eq:feps}. We therefore look for the solution to \eqref{eq:feps} under the form
\begin{align}
\label{def:g}
    f^\eps(t,x,v)= (f_{\rm app}^{\eps(M)}+g^\eps)(t,x,v) \, ,
\end{align}
where~$M$ will be chosen large enough at the end (of the order of~$|\ln (\eps)|$).   From a direct computation we   observe that $g^\eps$ solves
\begin{equation}
\label{eq:g}
  \left\{ \begin{aligned}
\de_t g^\eps+\eps^{-1}v\cdot \nabla_x g^\eps-\eps^{-2}Lg^\eps+\mathcal E^{\eps(M)}=\eps^{-1}\cL^\eps_{\rm app}(g^\eps)+\eps^{-1}\Gamma(g^\eps,g^\eps)\\
        g^\eps|_{t=0}=0 \, ,
\end{aligned}
        \right.
\end{equation}
where the error $\mathcal E^{\eps(M)}$ is defined as in~(\ref{def:ErrorM}) by 
\begin{equation}
\label{def:R}
    \mathcal E^{\eps(M)}\coloneqq\de_tf^{\eps(M)}_{{\rm app}}+\eps^{-1}\big(v\cdot \nabla_xf^{\eps(M)}_{{\rm app}}-\Gamma(f^{\eps(M)}_{{\rm app}},f^{\eps(M)}_{{\rm app}})\big)-\eps^{-2}Lf^{\eps(M)}_{{\rm app}}\,,
\end{equation}
and the linearization of $\Gamma$ at $f^{\eps(M)}_{\rm app}$ is given by 
\begin{align}\label{def:Lapp}
    \cL^\eps_{\rm app}(g):=\Gamma(f^{\eps(M)}_{\rm app},g)+\Gamma(g,f^{{\eps(M)}}_{\rm app})\,.
\end{align}
  Our goal is to prove that $g^\eps$ goes to zero as~$\eps$ goes to zero: this will require choosing~$M \sim |\ln(\eps)|$.  Let us indeed prove the following result.
\begin{proposition}\label{prop:solvegeps}
Let~$k>2$ be given. There is a constant~$C>1$ such that the following holds.   Choosing~$K$ large enough and~$\delta$ small enough,   there exists $\lambda >0$ such that setting $M={M(\eps)=\lceil\lambda|\ln (\eps)|\rceil}$ and~$T_\eps=\delta\eps^2/  M^6$, then there is a unique solution $g^\eps$ to~{\rm(\ref{eq:g})} on $[0,T_\eps]$ such that 
    \begin{align}
    \label{bd:nonlinear_prop}
   \forall \, 0\leq t \leq    T_\eps \, , \quad   \|g^\eps(t)\|_{{\mathcal X}^{1,k}}\ {+\eps\,\|\nabla_xg^\eps(t)\|_{{\mathcal X}^{1,k}} }\leq   {C\frac{M^\frac{11}2}{K\eps^3}t\left(  \frac{C_*  {M}^\frac32 t}{\eps^2}\right)^M}
   \leq \frac{\delta^{M+1}}{\eps \sqrt M} \, ,
    \end{align}
    where the last inequality holds for $\eps$ sufficiently small.    In particular, 
    \begin{align}
        \sup_{0\leq t\leq T_{\eps}}\left(
 \|g^\eps(t)\|_{\mathcal X^{1,k}}
 +\eps\|\nabla_x g^\eps(t)\|_{\mathcal X^{1,k}}
\right)
\longrightarrow 0,
\qquad \eps\to0 .
    \end{align}
    Consequently, for the full solution $
        f^\eps=f_{\rm app}^{\eps(M)}+g^\eps,$
 the macroscopic lower bounds obtained in Lemma~\ref{lem:macro_bounds} for
$f_{\rm app}^{\eps(M)}$ also hold for  $f^\eps$, possibly with a smaller
constant in \eqref{eq:app_bound_u}--\eqref{eq:app_bound_rho}.
\end{proposition}
Since~$    \sup_{t\leq T_\eps} \|g^\eps(t)\|_{{\mathcal X}^{1,k}}$ goes to zero in~${\mathcal X}^{1,k}$, then~$f^{\eps(M)}_{{\rm app}}(t)$ is a good approximation to the solution of the rescaled Boltzmann equation~\eqref{eq:feps} on that short time interval.

To solve the equation~(\ref{eq:g}) on~$[0,T_\eps]$, we shall use a standard fixed point argument, as given in the following lemma. 
\begin{lemma}
\label{cacciopoli+}
  Let~$X$ be a Banach space, let~${\mathcal L}$ be a  continuous linear map
from~$X$ to~$X$,  and let~${\mathcal B}$ be a bilinear map from~$X\times X$ to~$X$.
Let us define
\[
\|{\mathcal L}\|  := \sup_{\|x\|=1} \|{\mathcal L}x\|\quad\hbox{and}\quad
\|{\mathcal B}\| := \sup_{\|x\|=\|y\|=1} \|{\mathcal B}(x,y)\| \, .
\]
If~$\|{\mathcal L}\| <1$, then for any~$x_{0}$ in~$X$ such that
\[
\|x_{0}\|_{X}< \frac{(1-\|{\mathcal L}\| )^2} {4\|{\mathcal B}\| } 
\]
the equation
\[
x=x_{0}+{\mathcal L}x+{\mathcal B}(x,x)
\]
has a unique solution in the ball of center~$0$ and
radius~$\displaystyle \frac {1-\|{\mathcal L}\| }{2\|{\mathcal B}\|} $  and there is a constant~$C_0$ such that
$$
\|x\|\le C _0\|x_0\| \, .
$$
\end{lemma}
We are now ready to prove Proposition \ref{prop:solvegeps}.
\begin{proof}
 Let us define the semi-group~$U^\eps(t)$ associated with~$-\eps^{-1}v\cdot \nabla_x g+ \eps^{-2}Lg$. We rewrite~(\ref{eq:g}) as
\begin{equation}\label{eq:duhamel g}
\begin{aligned}
g^\eps(t) = -\int_0^t U^\eps(t-\tau ) \mathcal E^{\eps(M)} (\tau) \, \dd \tau& + \frac1\eps \int_0^t  U^\eps(t-\tau ) \cL^\eps_{\rm app}(g^\eps) (\tau)\, \dd \tau\\
& \qquad  + \frac1\eps \int_0^t  U^\eps(t-\tau ) \Gamma(g^\eps,g^\eps) (\tau)\, \dd \tau \, .
\end{aligned}
\end{equation}
In order to apply Lemma~\ref{cacciopoli+}
we introduce, as in~\cite{GT}, the following function space:
\begin{equation}\label{def:X1kT} 
\mathcal{X}_T^{1,k} := \Big\{f = f(t,x,v) \, / \,  f\in L^\infty( \mathbf{1}_{[0,T]}(t)\langle t \rangle^\frac14 \,\mathcal X^{1,k})
\Big\}\, ,
\end{equation}
equipped with the norm
$$
\|f\|_{\mathcal{X}_T^{1,k} }  := \sup_{[0,T]}\langle t \rangle^\frac14 \|f(t)\|_{\mathcal X^{1,k}} \, .
$$
The semi-group~$U^\eps(t)$   is uniformly continuous in~${\mathcal X}^{1,k}_T$ for all real numbers~$ k>2$ (see \cite{BU,Ukai} or~\cite[Lemma 3.1]{GT}, as well as Appendix~\ref{sec:boltzmann}): there holds for all~$f \in \mathcal{X}_T^{1,k} $ $$
\|U^\eps(t)f\|_{\mathcal{X}_T^{1,k} } \lesssim \|f\|_{\mathcal{X}_T^{1,k} }\, .
$$
In the following we shall control the first term of~\eqref{eq:duhamel g}
 in~${\mathcal X}^{1,k}_T$, prove a bicontinuity estimate for  the last one, and finally prove that the second one is a contraction as soon as~$K$ is large enough and~$T_\eps$ is small enough.   

Let us start by  studying the first term, involving~$\mathcal E^{\eps(M)}$. {We denote 
\begin{align}
    a_\eps(t):=\frac{C_*  {M}^\frac32 t}{\eps^2}.
\end{align} and we observe that, if $M$ is large enough, then 
\begin{equation}
\label{bd:aeps}
a_{\eps}(t)\leq        \frac{C_*\delta}{M^{9/2}} \leq
 \delta \quad \text{ for all } t\leq T_{\eps}\,. 
\end{equation} 
Then we rewrite} the  estimate~\eqref{eq:error_bound} as
$$
    \| \cE^{\eps(M)}(t) \|_{\mathcal{X}^{1,k}} \leq \frac{C_kC_*  {M}^\frac32}{K \eps^3} \sum_{j=M}^{2M} (j+2)^4 (a_\eps(t))^j \, .
$$
Following similar computations to the proof of Corollary~\ref{cor:bound on the approximation}, for all~$t\leq T_\eps$, 
{we get the upper bound \begin{align}\label{eq:EepsM} 
\int_0^{t} \| \cE^{\eps(M)}(\tau ) \|_{\mathcal{X}^{1,k}} \, d\tau
\leq\,  t \frac{C_k C_* \sqrt{M}}{K \eps^3} \sum_{j=M}^{2M} (j+2)^4\left(a_\eps(t)\right)^j\lesssim 
\frac{M^{\frac{11}{2}}}{K\eps^3}\,
t\,
\left(a_\eps(t)\right)^M \leq \,  \frac{M^\frac{11}2}{K\eps^3} \delta^{M }t \, .
\end{align}}
Now let us turn to the   two other terms in the Duhamel formulation~(\ref{eq:duhamel g}).  We define
\begin{equation}\label{eq:defPsieps}
\Psi^\eps(t)(h_1,h_2):= \frac1\eps \int_0^t  U^\eps(t-\tau )\big(\Gamma(h_1,h_2) +\Gamma(h_2,h_1)\big)(\tau)\, \dd \tau\, .
\end{equation}
As recalled in   Appendix~\ref{sec:boltzmann},    for any~$T>0$ and~$k>2$
  \begin{equation}\label{eq:estimate Psieps}
  \| \Psi^\eps(t)(h_1,h_2) \|_{\mathcal{X}_T^{1,k} } \lesssim   \|h_1\|_{\mathcal{X}_T^{1,k} }  \|h_2\|_{\mathcal{X}_T^{1,k} } \, .
\end{equation} 
This deals with the bilinear term in the fixed point argument since we infer in particular that
\begin{equation}\label{eq:estimate bilinear}
 \frac1\eps\Big\| \int_0^t  U^\eps(t-\tau ) \Gamma(g^\eps,g^\eps) (\tau)\, \dd \tau \Big\|_{\mathcal{X}_T^{1,k}  }\lesssim   \|g^\eps\|_{\mathcal{X}_T^{1,k}  }^2\, .
\end{equation}
 For the linear term in~\eqref{eq:duhamel g},  estimate~(\ref{eq:estimate Psieps}) is unfortunately not sufficient as we need   the operator to be a contraction: the estimate  provides indeed
 $$
  \frac1\eps\Big\| \int_0^t  U^\eps(t-\tau )    \cL^\eps_{\rm app}(g^\eps) (\tau)\, \dd \tau \Big\|_{\mathcal{X}_T^{1,k} }  \lesssim   \|f^{\eps(M)}_{\rm app}\|_{\mathcal{X}_T^{1,k} }  \|g^\eps\|_{\mathcal{X}_T^{1,k} } 
 $$
 but~$ \|f^{\eps(M)}_{\rm app}\|_{\mathcal{X}_T^{1,k} } $ is not small so this term does not directly satisfy the assumptions of Lemma~\ref{cacciopoli+}. 
 However one can invoke Lemma~\ref{lem:Gamma better bound}
according to which
$$
  \frac1\eps\Big\| \int_0^t  U^\eps(t-\tau )    \cL^\eps_{\rm app}(g^\eps) (\tau)\, \dd \tau \Big\|_{\mathcal{X}_T^{1,k} } \lesssim   ( \eps  +  {\sqrt T} )  \|f^{\eps(M)}_{\rm app}\|_{\mathcal{X}_T^{1,k} }  \|g^\eps\|_{\mathcal{X}_T^{1,k} }   \, .
$$
 Then thanks to Corollary~\ref{cor:bound on the approximation} we infer that  
   $$
   \frac1\eps\Big\| \int_0^t  U^\eps(t-\tau )    \cL^\eps_{\rm app}(g^\eps) (\tau)\, \dd \tau \Big\|_{\mathcal{X}_T^{1,k} }   \lesssim \Big( \frac1{K}
 + \frac {{\sqrt T}}{K\eps }\big) \|g^\eps\|_{\mathcal{X}_T^{1,k} } 
 $$
 It follows that   for~$K $ large enough depending only on universal constants, 
 if~$M = \lfloor \lambda |\ln (\eps)|\rfloor$ then   the linear term   is a contraction in~$\mathcal{X}_{T}^{1,k}  $ for all~$T \leq T_\eps$.   
 
 Lemma~\ref{cacciopoli+}, along with~\eqref{eq:EepsM}, finally implies that there is a unique solution to~(\ref{eq:duhamel g}), which satisfies
$$
\forall T \leq T_\eps \, , \quad \|g\|_{\mathcal{X}_{T }^{1,k}  }  {\lesssim 
\frac{M^{\frac{11}{2}}}{K\eps^3}\,
T\,
\left(a_\eps(T)\right)^M }\lesssim \frac{M^\frac{11}2}{K\eps^3} \delta^{M }T  \, .$$
This proves the first bound in \eqref{bd:nonlinear_prop} upon relabelling $T$ with $t$.

{It remains to control one spatial derivative, for which we denote
\[
        h^\eps:=\eps\nabla_x g^\eps .
\]
Clearly the $\nabla_x$  commutes with  the semigroup $U^\eps(t)$, and we also have
\[
\nabla_x\Gamma(f,g)=\Gamma(\nabla_x f,g)+\Gamma(f,\nabla_x g).
\]
Thus, from~\eqref{eq:duhamel g} with the notation in \eqref{eq:defPsieps}, we obtain

\begin{equation}\label{eq:duhamel_eps_nabla_x_g}
\begin{aligned}
h^\eps(t)
&=
-\int_0^t U^\eps(t-\tau)\,
\eps\nabla_x\mathcal E^{\eps(M)}(\tau)\,\dd\tau+\Psi^\eps(f_{\rm app}^{\eps(M)},h^\eps)(t)
\\
&\quad
+\Psi^\eps(\eps\nabla_x f_{\rm app}^{\eps(M)},g^\eps)(t)
+\Psi^\eps(h^\eps,g^\eps)(t).
\end{aligned}
\end{equation}
Arguing as done to get the profile estimates in Lemma~\ref{lem:algebra} and Lemma~\ref{lem:recurrence_bounds},
since the $\eps$ compensates for the loss generated by $\nabla_x$, it is not hard to deduce that 
\[
\|\eps\,\nabla_x f_{\rm app}^{\eps(M)}\|_{\mathcal X_T^{1,k}}
\lesssim \frac1{K\eps},
\]
and therefore
\begin{equation}\label{eq:grad_EepsM}
\int_0^t
\|\eps\,\nabla_x\mathcal E^{\eps(M)}(\tau)\|_{\mathcal X^{1,k}}\,d\tau
\lesssim
\frac{M^{\frac{11}{2}}}{K\eps^3}
t\left(a_\eps(t)\right)^M .
\end{equation}
We can then bound the linear in $h^\eps$ terms in~\eqref{eq:duhamel_eps_nabla_x_g} by 
\[
C\left(\frac1K+\frac{\sqrt T}{K\eps}\right)
\|h^\eps\|_{\mathcal X_T^{1,k}},
\]
and this follows by arguing as done for the linear terms in $g^\eps$. Similarly, the terms containing $\eps\nabla_x f_{\rm app}^{\eps(M)}$ are bounded as above with
$\|g^\eps\|_{\mathcal X_T^{1,k}}$ instead of $\|h^\eps\|_{\mathcal X_T^{1,k}}$. Finally, by~\eqref{eq:estimate Psieps}, the terms containing both
$h^\eps$ and $g^\eps$ are bounded by
\[
C\|g^\eps\|_{\mathcal X_T^{1,k}}\|h^\eps\|_{\mathcal X_T^{1,k}} .
\]
We can then apply Lemma \ref{cacciopoli+} in the  space $\mathcal{Y}_T^{1,k}$ equipped with the norm
\begin{align}
    \|f\|_{\mathcal{Y}_{T}^{1,k}}:=\|f\|_{\mathcal{X}_{T}^{1,k}}+\eps \|\nabla_x  f\|_{\mathcal{X}_{T}^{1,k}}.
\end{align}
Here, we can use the already proved bound on $g^\eps$, and taking $K$ large, $\delta,\eps$ small, together
with~\eqref{eq:grad_EepsM} we get
\[
\|h^\eps\|_{\mathcal X_T^{1,k}}
\lesssim
\frac{M^{11/2}}{K\eps^3}
T\left(a_\eps(T)\right)^M ,
\]
and the derivative estimate follows by taking $T=t$.

Finally, it remains to show that the macroscopic lower bounds of
Lemma~\ref{lem:macro_bounds} are stable under the correction
\[
        f^\eps=f_{\rm app}^{\eps(M)}+g^\eps .
\]
By~\eqref{bd:nonlinear_prop} and  elementary moment bounds, we have
\begin{equation}\label{eq:moment_correction}
\big(|\rho[g^\eps]|+|u[g^\eps]|+|\theta[g^\eps]|
+\eps |\nabla_x\cdot u[g^\eps]|\big)(t,x)
\lesssim
\frac{t}{K\eps^3}M^{11/2}\big(a_\eps(t)\big)^M .
\end{equation}
We then have to compare the right-hand side of~\eqref{eq:moment_correction} with the lower bounds in
Lemma~\ref{lem:macro_bounds}. For the velocity, appealing to \eqref{bd:aeps} we see that
\[
\frac{|u[g^\eps](t,x)|}{t/\eps^3}
\lesssim
\frac1K M^{11/2}\big(a_\eps(t)\big)^M
\le
\frac1K M^{11/2}
\left(\frac{C\delta}{M^{9/2}}\right)^M\ll 1
\]
upon choosing $\delta,\eps$ sufficiently small. 
Analogous estimates holds for all other macroscopic quantities and we avoid detailing the precise power counting of powers of $M$ and $\delta$ involved. Thus, we conclude that choosing \(\delta\) small enough, then \(K\)
large enough, and finally \(\eps\) small enough, all the correction terms are smaller than (at least) a half
of the corresponding lower bounds in Lemma~\ref{lem:macro_bounds}. Therefore the bounds
\eqref{eq:app_bound_u}--\eqref{eq:app_bound_rho} remain valid for
\(f^\eps=f_{\rm app}^{\eps(M)}+g^\eps\), possibly with a smaller constant. This ends the proof of Proposition~\ref{prop:solvegeps}.}
\end{proof}
{Finally, we note that the result in Proposition \ref{prop:solvegeps}, combined with the statement in Lemma \ref{lem:macro_bounds}, gives the result claimed in Theorem \ref{thm:result}.}
\section{Further directions: general scale-invariant data}
\label{sec:3D_perspectives}
In this paper we have considered the explicit incompressible target of the Lamb--Oseen vortex in view of its relevance in 2D Navier--Stokes. Nevertheless, the mechanism behind the instability is
not tied to having an explicit formula for the initial data. The really essential feature is the $-1$-homogeneity of the
underlying fluid velocity, which forces the regularized core to have size
$|u_0^\eps|\sim \eps^{-1}$ on $|x|\sim\eps$. This creates an initial layer in which
the kinetic transport and the quadratic collision are dominant and lead the instability.

We briefly indicate how one might extend the construction to a more general setting. Let $u_0$ be a divergence free, $-1$-homogeneous vector field in $\mathbb R^d$, $d=2$ or $3$, smooth away from the origin. A possible regularization in $d=3$ is to follow the natural heat evolution for a short time, truncate and take the Leray projector $\mathbb{P}$, namely
\begin{equation}
    u_0^\eps(x):=\mathbb{P}\left(\chi(\eps^\gamma |x|)\,
    e^{(K\eps)^2\Delta}u_0\right),
\end{equation}
whereas in 2D it is enough to regularize and truncate directly the streamfunction for instance. One may then consider the
well-prepared kinetic datum
\[
    f_{\rm in}^\eps(x,v):=u_0^\eps(x)\cdot v\sqrt\mu(v).
\]
For $d=3$ the functional setting should be strengthened, for instance by replacing
$\mathcal X^{1,k}$ with
\[
\mathcal X_{3D}^{{3/2},k}
:=
\left\{
f:\ \|f(\cdot,v)\|_{H^{3/2}_x(\mathbb R^3)\cap W_x(\mathbb R^3)}
\in L_v^{\infty,k}
\right\},
\]
so that the required algebra properties and pointwise embeddings remain available.
At the level of the formal hierarchy, the classes $\mathcal P_\eps(n)$ and the
iterative construction should then have direct analogues, with dimension dependent
changes in the Sobolev and Wiener estimates.

The macroscopic comparison is less explicit for a general profile. For small
$-1$-homogeneous data in three dimensions, Brandolese's asymptotic expansion
\cite{brandolese2009} provides precise information on the short-time behavior of
the corresponding Navier--Stokes solution. On the other hand, for large
$-1$-homogeneous data, Jia and \v{S}ver\'ak \cite{JiaSverak} constructed forward
self-similar solutions and obtained local-in-space estimates near the initial time,
identifying the leading behavior after subtracting the heat evolution of the initial
datum. For the present purpose one would need a version of these expansions, with
remainder estimates uniform in the regularization parameter $\eps$, in the core region $|x|\sim\eps$. In particular, by iterating the Duhamel formulation
\[
    u(t,x)
    =
    e^{\nu t\Delta}u_0^\eps(x)
    -
    \int_0^t e^{\nu (t-\tau)\Delta}\mathbb P
    \nabla\cdot \bigl(u(\tau)\otimes u(\tau)\bigr)(x)\,d\tau \, ,
\]
 it seems quite feasible (at least for small data) to  rigorously justify a short time and local in space expansion of the form 
\[
    u(t,x)
    =
    u_0^\eps(x)
    +
     t\Big(\nu\Delta u_0^\eps
      -\mathbb P(u_0^\eps\cdot\nabla u_0^\eps)\Big)(x)
    +\text{higher order terms} \, ,
\]
at least in regimes where the regularization is fixed and $t\to0$. 
Evaluated at
$|x|\sim\eps$, the first correction is typically of size $t\eps^{-3}$, unless the
profile satisfies special cancellation\footnote{Note that the order $t$ vanishes for stationary solutions of the Navier--Stokes equations.}. We observe that in 2D, a generic data introduces logarithmic corrections \cite{brandolese2009}, but one can easily adapt accordingly the lower bounds. Thus, this  suggests that at leading order the difference with respect to the initial data resembles the leading order term in the Lamb--Oseen case.
On the kinetic side, the first correction
\[
    f_1^\eps
  :  =
    \eps^{-1}\Big(
        -v\cdot\nabla_x f_0^\eps
        +\Gamma(f_0^\eps,f_0^\eps)
    \Big)
\]
still has vanishing hydrodynamic projection. Thus the macroscopic velocity does not acquire
the Navier--Stokes first-order correction at order $t$. The first nontrivial
macroscopic kinetic correction appears at the next order,  which gives contributions of the form
\[
    \eps^{-2}\Big(
        c_1 (u_0^\eps\cdot\nabla)u_0^\eps
        +c_2\nabla |u_0^\eps|^2
    \Big),
\]
up to the linear viscous contribution. This is because the fundamental properties that we used on the collision operator are independent of the dimension. Unless these terms satisfy additional
cancellations, the same compressible initial-layer mechanism should persist.

These remarks suggest a possible extension of the instability mechanism to more
general scale-invariant data. They also raise an interesting open problem in three
dimensions. Large forward self-similar Navier--Stokes solutions are known to be
closely related to nonuniqueness scenarios through the spectral mechanism
proposed by Jia and \v{S}ver\'ak in \cite{JiaSveraknonuniqueness}, numerically verified in \cite{Guillod} and recently  a computer assisted proof was announced  in \cite{Hounon}. It would be interesting to understand whether a
regularized kinetic evolution, which is unique at the mesoscopic level in a suitable
perturbative regime, selects a particular macroscopic branch after the initial kinetic
layer. At present this remains speculative though. Even for the explicit Lamb--Oseen vortex, the
precise transition from the initial layer $t\ll\eps^2$ to the hydrodynamic regime~$t\gg\eps^2$ seems an interesting problem to address, and in this case there are no issues with the uniqueness of the solution.

\appendix
\section{On the Boltzmann equation} 
\subsection{The Cauchy problem in $\mathcal{X}^{1,k}_T$}
\label{sec:boltzmann}
This section gathers some estimates that have been used in the paper, with indications of proofs when they are not easily found in the literature: as explained in the introduction of this paper, the functional setting chosen to study the rescaled Boltzmann equation \eqref{eq:feps}, namely initial data in $W_x \cap H^1_x$ in the spatial variable,  is indeed not the most classical one.   We do not claim any originality nor optimality in the results presented here, and refer for instance to~\cite{DLX,GMM} for results in the~$L^\infty_x$ scaling in space for the Boltzmann equation~\eqref{eq:B} in three space dimensions.

Here we borrow results and techniques from~\cite{GT} giving  estimates on the Duhamel formulation~(\ref{eq:duhamel g}). More specifically we shall be inspired by the proof of~\cite[Lemma~3.7]{GT} to sketch the proof of the following result. 

\begin{lemma}\label{lem:Gamma better bound}
Recalling definition~\eqref{def:X1kT} for the space~$ \mathcal{X}^{1,k}_T$, the bilinear operator~$\Psi^{\eps}(t)$ defined in~\eqref{eq:defPsieps} satisfies
$$
\forall T\geq 0 \,  , \quad \|\Psi^{\eps}(t)(h_1,h_2)\|_{ \mathcal{X}^{1,k}_T} \lesssim   ( \eps  +  \sqrt T ) \|h_1 \|_{  \mathcal{X}^{1,k}_T}  \|h_2 \|_{  \mathcal{X}^{1,k}_T}     \, .
$$
\end{lemma}
\begin{proof}

We follow the argument of~\cite[Lemma 3.7]{GT}. {The main idea, following Grad's decomposition, is to first split $L=-\nu(v)+K$ where $-\nu(v)$ is the collision frequency and $K$ is compact. Then, one can define $A^\eps=-\eps^{-2}(\eps v\cdot \nabla_x+\nu(v))$ and write $U^\eps(t)$ as 
$$U^\eps(t)=e^{tA^\eps}+\frac{1}{\eps^2}\int_0^te^{(t-t')A^\eps}KU^\eps(t')dt',
$$
where the semigroup $e^{tA^\eps}$ admits an explicit representation and its action is bounded by $e^{-\nu_0 t/\eps^2}$ in a suitable sense. Relying also on the Ellis and Pinsky \cite{Ellis-Pinsky} spectral decomposition of $U^\eps(t)$, in \cite{GT} it is proved that }
$\Psi^\eps(t)(h_1,h_2)$ can be split in two parts:
$$
\Psi^\eps(t)(h_1,h_2) = \Psi^{\eps,1}(t)(h_1,h_2) + \Psi^{\eps,2}(t)(h_1,h_2) \, ,
$$ 
where the first one   is easily bounded  in~$ \mathcal{X}_T^{1,k} $  as follows:
$$\|\Psi^{\eps,1}(t)(h_1,h_2)\|_{ \mathcal{X}_T^{1,k} } \lesssim  \eps(  \|(1+|v|)^{-1} \Gamma(h_1,h_2)
\|_{ \mathcal{X}_T^{1,k} } + \|(1+|v|)^{-1} \Gamma(h_2,h_1)
\|_{ \mathcal{X}_T^{1,k} } )\,   .
$$
Then using the fact that~$H^1_x \cap W_x$ is an algebra (instead of choosing the function space~$H^ {\ell}_x$  with~$\ell>1$ as in~\cite{GT}) we find easily
$$
\|\Psi^{\eps,1}(t)(h_1,h_2)\|_{\mathcal{X}_T^{1,k} } \lesssim  \eps \|h_1 \|_{\mathcal{X}_T^{1,k} }  \|h_2 \|_{ \mathcal{X}_T^{1,k} }   \, .
$$
The second one is   estimated in~$L^\infty([0,T];  H^ {1}_x L^2_v)$ thanks to the following estimate (see~\cite[Lemma~3.2]{GT}):
$$
\frac{1}{\eps}\|U^\eps(t)(1-\bold{P})f\|_{H^1_xL^2_v} \lesssim 
\frac{1}{t^{\frac12} \langle t \rangle^{\frac12}} (\|f\|_{H^1_xL^2_v} + \|f\|_{L^2_vL^1_x})   \, .$$
 Then  following the proof of~\cite{GT} one has  
$$
\sup_{[0,T]}\langle t \rangle^\frac14 \|\Psi^\eps(t) (h_1,h_2) \|_{(H^1_x \cap W_x) L^{\infty,k}_v} \lesssim  \sup_{[0,T]} \, \langle t \rangle^{\frac14} \int_0^t \frac{1}{(t-t')^\frac12 \langle t-t' \rangle^\frac12} \frac{1}{\langle t' \rangle^\frac12} \, dt'   \|h_1\|_{ \mathcal{X}_T^{1,k}}   \|h_2\|_{\mathcal{X}_T^{1,k}}
$$
 so using again the fact that~$W_x \cap L^\infty_x$ is an algebra we obtain
$$
\begin{aligned}
  \| \Psi^{\eps,2}(t)(h_1,h_2)\|_{\mathcal{X}_T^{1,k} } &\lesssim\sqrt T \|h_1 \|_{ {\mathcal X}^{1,k}_T}  \|h_2 \|_{ {\mathcal X}^{1,k}_T} \,   .
\end{aligned}$$
Lemma~\ref{lem:Gamma better bound}
is proved.  
\end{proof}
   
 \subsection{Exponential Weight Bounds for the Collision Operator}\label{sec:appendix_collision}

In this appendix, we detail the pointwise exponential weight bounds for the nonlinear operator with hard-sphere collisions. We recall that~$\Gamma(f,g) = \mu^{-1/2}\mathcal{Q}(\mu^{1/2}f, \mu^{1/2}g)$.

\begin{lemma}\label{lem:Gamma_exponential_bound}
Let $\beta \ge 1/8$ and define the weight $\mu_\beta(v) = e^{-\beta |v|^2/4}$. Suppose that $F_1, F_2$ are functions satisfying $|F_1(v)| \le \mathfrak{F}_1 \mu_\beta(v)$ and $|F_2(v)| \le \mathfrak{F}_2 \mu_\beta(v)$ for some constants $\mathfrak{F}_1,\mathfrak{F}_2>0$. Then there exists a universal constant $C > 0$ such that
\begin{equation}
 \big   |\Gamma(F_1, F_2)(v)\big| \leq C \mathfrak{F}_1\mathfrak{F}_2 \langle v \rangle \mu_\beta(v) \, .
\end{equation}
\end{lemma}
\begin{proof}
For hard-sphere collisions, unpacking the definition of $\Gamma$ gives
\begin{equation}
    \Gamma(F_1, F_2)(v) = \int_{\RR^2 \times \mathbb{S}^1} |v-v_*| \left[ F_1(v'_*) F_2(v') - F_1(v_*) F_2(v) \right] \sqrt{\mu(v_*)} \dd\sigma \dd v_* \, .
\end{equation}
We estimate the gain and loss terms together. By the assumptions on $F_1$ and $F_2$, we have:
\begin{equation}
    |F_1(v'_*) F_2(v')| \leq \mathfrak{F}_1 \mathfrak{F}_2 \mu_\beta(v'_*) \mu_\beta(v') = \mathfrak{F}_1 \mathfrak{F}_2 e^{-\frac{\beta}{4}(|v_*'|^2 + |v'|^2)} \, .
\end{equation}
By the conservation of kinetic energy $|v_*'|^2 + |v'|^2 = |v_*|^2 + |v|^2$, we get
\begin{equation}
    |F_1(v'_*) F_2(v')| \leq \mathfrak{F}_1 \mathfrak{F}_2 \mu_\beta(v_*) \mu_\beta(v) \, .
\end{equation}
The exact same pointwise bound trivially holds for the loss term $F_1(v_*) F_2(v)$. Inserting these into the integral and applying the triangle inequality, we get:
\begin{align}
    |\Gamma(F_1, F_2)(v)| 
    &\leq 2 \mathfrak{F}_1 \mathfrak{F}_2 \mu_\beta(v) \int_{\RR^2 \times \mathbb{S}^1} |v-v_*| \mu_\beta(v_*) \sqrt{\mu(v_*)} \dd\sigma \dd v_* \, .
\end{align}
Using the elementary inequality $|v-v_*| \le \langle v \rangle \langle v_* \rangle$, since $\beta\geq 1/8$ we get
\begin{align}
    |\Gamma(F_1, F_2)(v)| 
    &\leq 2 \mathfrak{F}_1 \mathfrak{F}_2 \langle v \rangle \mu_\beta(v) \int_{\RR^2 \times \mathbb{S}^1} \langle v_*\rangle e^{-9|v_*|^2/32} \dd\sigma \dd v_* \, ,
\end{align}
which yields the desired bound after integrating above.
\end{proof}

\subsection{On the collision constant ${\sf c}_1$}\label{sct:collision constant}
Here we verify that the constant ${\sf c}_1$ appearing in \eqref{def:Buin} does not
vanish for the hard-sphere collision operator. Below $C\neq 0$ is a constant that can change from line to line. Evaluating the tensor $\mathcal{B}^\eps(u)$ in \eqref{def:Buin} at  $u=e_1$, we see that
\begin{equation}\label{eq:c1_traceless_test}
    \mathsf c_1
    =C
    \eps(\mathcal B^\eps_{11}(e_1)-\mathcal B^\eps_{22}(e_1))
    =C
    \int_{\mathbb R^2}\underbrace{(v_1^2-v_2^2)}_{:=\phi(v)}\,
    Q(v_1\mu,v_1\mu)(v)\,dv .
\end{equation}
 By the weak formulation of the Boltzmann
collision operator and by symmetrization in \((v,v_*)\), we have
\[
\mathsf c_1
=
C\int |v-v_*|\mu(v)\mu(v_*)v_1v_{*,1}
\bigl[\phi(v')+\phi(v_*')-\phi(v)-\phi(v_*)\bigr]
\,d\sigma\,dv_*\,dv ,
\]
for some $C\neq0$.
Setting
\[
p=\frac{v+v_*}{2},\qquad w=v-v_*,
\]
we have
\[
v_1v_{*,1}=p_1^2-\frac14w_1^2.
\]
and, since \(p\) is conserved by collisions, 
\[
v=p+\frac w2,\qquad v_*=p-\frac w2,\qquad
v'=p+\frac{w'}2,\qquad v_*'=p-\frac{w'}2,
\]
where \(w'=v'-v_*'=|w|\sigma\). Moreover, note that for any $a$ we get
\[
\phi(p+a)+\phi(p-a)=2\phi(p)+2\phi(a).
\]
Therefore
\[
\phi(v')+\phi(v_*')-\phi(v)-\phi(v_*)
=
\frac12\bigl(\phi(w')-\phi(w)\bigr),
\]
and, up to changing $C\neq0$,
\[
\mathsf c_1
=
C\int |w|\tilde{\mu}(p)\sqrt{\mu}(w)
\left(p_1^2-\frac14w_1^2\right)
\bigl(\phi(w')-\phi(w)\bigr)
\,d\sigma\,dp\,dw ,
\] 
where $\tilde{\mu}(p)=e^{-|p|^2}/(\pi\sqrt{2\pi})$.
Since $w'=|w|\sigma$, we know that
\[
\int_{\mathbb S^1}\phi(w')\,d\sigma=0.
\]
Therefore
\[
\int_{\mathbb S^1}(\phi(w')-\phi(w))\,d\sigma
=
-|\mathbb S^1|\phi(w).
\]
The contribution of the $p_1^2$ term vanishes by  symmetries in $w$.
Consequently,
\[
\mathsf c_1
=
C\int_{\mathbb R^2}
|w|\sqrt{\mu}(w)w_1^2(w_1^2-w_2^2)\,dw,
\]
for another constant $C\neq0$. Passing to polar coordinates gives
\[
|\mathsf c_1|
=
|C|
\left(\int_0^\infty r^6\sqrt{\mu}(r)\,dr\right)
\left(\int_0^{2\pi}
\cos^2\theta(\cos^2\theta-\sin^2\theta)\,d\theta\right)>0
\]
as desired.

{\bf Acknowledgements.}  $ $ MD was supported by the Swiss National Science Foundation (SNF Ambizione grant PZ00P2\_223294). IG was  supported by the BOURGEONS project, grant ANR-23-CE40-0014-01 of the French National Research Agency (ANR).

\end{document}

%% file: macros.tex
\def\dd{{\rm d}}

\def\eps{\varepsilon}
\def\epsilon{\varepsilon}
\def\e{{\rm e}} 
\def\de{{\partial}}

\def\RR {\mathbb{R}}
\def\R {\mathbb{R}} 

\newcommand{\cE}{\mathcal{E}}

\newcommand{\cL}{\mathcal{L}}

\def\virgp{\raise 2pt\hbox{,}}